\documentclass[12pt]{article}
\usepackage{amsmath,amsfonts,amsthm,amssymb}
\usepackage{typearea}
\usepackage{fancybox}

\newcommand{\stsets}[1]{\mathbb{#1}}
\newcommand{\R}{\stsets{R}}
\newcommand{\M}{\stsets{M}}

\newcommand{\Z}{\stsets{Z}}

\theoremstyle{definition}
\newtheorem{definition}{Definition}
\theoremstyle{remark}
\newtheorem{remark}{Remark}
\newtheorem{example}{Example}

\newlength{\querylen}
\theoremstyle{plain}
\newtheorem{theorem}{Theorem}
\newtheorem{proposition}{Proposition}

\newtheorem{corollary}{Corollary}

\renewcommand{\P}{\mathbf{P}}

\DeclareMathOperator{\E}{{\bf E}}

\DeclareMathOperator{\supp}{supp}

\newcommand{\salg}{$\sigma$-algebra\ }
\newcommand{\zero}{\ensuremath{\mathbf{0}}}

\newcommand{\thru}{,\dotsc,}
\DeclareMathOperator{\bydef}{:=}

\renewcommand{\ln}{\log}

\renewcommand{\epsilon}{\varepsilon}
\renewcommand{\phi}{\varphi}

\newcommand{\ssp}{\hspace{0.5pt}}
\newcommand{\seg}{see, \hbox{e.\ssp g.,}\ }
\newcommand{\ie}{\hbox{i.\ssp e.}\ }

\newcommand{\pgf}{\hbox{p.\ssp g.\ssp f.}\ }
\newcommand{\pgfl}{\hbox{p.\ssp g.\ssp fl.}\ }

\newcommand{\sB}{\mathcal{B}}

\newcommand{\sM}{\mathcal{M}}
\newcommand{\sN}{\mathcal{N}}
\newcommand{\sF}{\mathcal{F}}

\newcommand{\sV}{\mathcal{V}}
\newcommand{\sX}{\mathcal{X}}
\newcommand{\sY}{\mathcal{Y}}

\newcommand{\EE}{\mathbb{E}}

\renewcommand{\SS}{\mathbb{S}}

\newcommand{\RR}{\mathbb{R}}
\newcommand{\ZZ}{\mathbb{Z}}

\newcommand{\sas}{\ensuremath{\mathrm{St\alpha S}}}
\newcommand{\das}{\ensuremath{\mathrm{D\alpha S}}}
\newcommand{\bmp}{\ensuremath{\mathrm{BM}_+(\sX)}}

\newcommand{\deq}{\overset{\mathcal{D}}{=}}

\newcommand{\Bin}{\ensuremath{\mathsf{Bin}}}
\newcommand{\Po}{\ensuremath{\mathsf{Po}}}
\newcommand{\Sib}{\ensuremath{\mathsf{Sib}}}

\begin{document}
\title{Branching-stable point processes}
\author{Giacomo Zanella\thanks{Department of Statistics, University of Warwick, Coventry CV4 7AL,
    U.K. Email: G.Zanella@warwick.ac.uk} 
\and Sergei Zuyev\thanks{Department of Mathematical Sciences, 
Chalmers University of Technology 
and University of Gothenburg,
412 96 Gothenburg, Sweden.
Email: sergei.zuyev@chalmers.se}}
\date{\today}
\maketitle
\begin{abstract}
  The notion of stability can be generalised to point processes by
  defining the scaling operation in a randomised way: scaling a
  configuration by $t$ corresponds to letting such a configuration
  evolve according to a Markov branching particle system for -$\log t$
  time.  We prove that these are the only stochastic operations
  satisfying basic associativity and distributivity properties and we
  thus introduce the notion of branching-stable point processes.  We
  characterise stable distributions with respect to local branching as
  thinning-stable point processes with multiplicities given by the
  quasi-stationary (or Yaglom) distribution of the branching process
  under consideration.  Finally we extend branching-stability to
  random variables with the help of continuous branching (CB)
  processes, and we show that, at least in some frameworks,
  $\mathcal{F}$-stable integer random variables are exactly Cox (doubly
  stochastic Poisson) random variables driven by corresponding
  CB-stable continuous random variables.

  
\end{abstract}

\section{Introduction}
\label{sec:introduction}

The concept of stability is central in Probability theory: it
inevitably arises in various limit theorems involving scaled sums of
random elements. Recall that a random vector $\xi$ (more generally, a
random element in a Banach space) is called \emph{strictly
  $\alpha$-stable} or \sas, if
\begin{equation}\label{definition}
  t^{1/\alpha}\xi'+(1-t)^{1/\alpha}\xi''\deq\xi\qquad\text{for all}\ t\in [0,1],
\end{equation}\
where $\xi'$ and $\xi''$ are independent copies of $\xi$ and $\deq$
denotes equality in distribution. When a limiting distribution for the
sum of $n$ independent vectors scaled by $n^{1/\alpha}$ exist, it must
be \sas, since one can divide the sum into first $tn$ and the last
$(1-t)n$ terms which, in turn, also converge to the same law. This
simple observation gives rise to the defining
identity~(\ref{definition}). It is remarkable that the same argument
applies to random elements in much more general spaces where two
abstract operations are defined: a sum and a scaling by positive
numbers which should satisfy mild associativity, distributivity and continuity
conditions, \ie in a cone, see~\cite{DavMolZuy:08}. For instance, a
random measure $\xi$ on a general complete separable metric space is
called strictly $\alpha$-stable if identity~(\ref{definition}) is
satisfied, where the summation of measures and their multiplication by
a number are understood as the corresponding arithmetic operations on
the values of these measures on every measurable set.
Stable measures are the only class of measures which arise as a weak
limit of scaled sums of random measures.

Since the notion of stability relies on multiplication of a random
element by a number between 0 and 1, integer valued random variables
cannot be \sas. Therefore Steutel and van Harn in their pioneering
work \cite{ste:har79} defined a stochastic operation of \emph{discrete
  multiplication} on positive integer random variables and
characterised the corresponding \emph{discrete $\alpha$-stable} random
variables. In a more general context, the discrete multiplication
corresponds to the \emph{thinning operation} on point processes when a
positive integer random variable is regarded as a trivial point
process on a phase space consisting of one point (so it is just the
multiplicity of this point). This observation leads to the notion of
\emph{thinning stable} or discrete $\alpha$-stable point processes (notation:
\das) as the processes $\Phi$ which satisfy
\begin{equation}\label{das}
  t^{1/\alpha}\circ\Phi'+(1-t)^{1/\alpha}\circ\Phi''
  \deq\Phi\qquad\text{for all}\ t\in [0,1],
\end{equation}\
when multiplication by a $t\in[0,1]$ is replaced by the operation
$t\circ$ of \emph{independent thinning} of their points with the
retention probability $t$. The \das\ point processes are exactly the
processes appearing as a limit in the superposition-thinning schemes
(see \cite[Ch.~8.3]{Kal:83}) and their full characterisation was given
in~\cite{DavMolZuy:11}.

In its turn, a thinning could be thought of as a particular case of a
branching operation where a point either survives with
probability $t$ or is removed with the complimentary probability. This
observation leads to a new notion of discrete stability for point
processes by considering a more general branching operation based on a
subcritical Markov branching process $(Y_t)_{t>0}$ with generator
semigroup $\mathcal{F}=(F_t)_{t\geq 0}$, satisfying
$Y_0=1$. Following Steutel and Van Harn~\cite{SHV:82} who considered
the case of integer-valued random variables, we denote this operation
by $\circ_\mathcal{F}$. In this setting when a point process is
``multiplied" by a real number $t\in(0,1]$, every point is replaced by
a collection of points located in the same position of their
progenitor. The number of points in the collection is a random variable
distributed as $Y_{-\ln t}$. This operation preserves distributivity
and associativity with respect to superposition and generalises the
thinning operation. In Section~\ref{f_stab_pp} we study stable point
processes with respect to this branching operation $\circ_\mathcal{F}$
calling them $\mathcal{F}$-stable point processes. We show that
$\mathcal{F}$-stable point processes are essentially \das\ processes
with multiplicities which follow the limit distribution $Y_\infty$ of
the branching process $Y_t$ conditional on its survival (Yaglom
distribution) and we deduce their further properties.

In a broader context, given an abstract
associative and distributive stochastic operation $\bullet$ on point
processes, a process $\Phi$ is stable with respect to $\bullet$ if
and only if
\begin{displaymath}
  \forall n\in\mathbb{N}\hspace{1.5mm}\exists
  c_n\in[0,1]\hspace{1.5mm}:\hspace{1.5mm}\Phi \stackrel{\mathcal{D}}
  = c_n\bullet (\Phi^{(1)}+...+\Phi^{(n)}), 
\end{displaymath}
where $\Phi^{(1)},...,\Phi^{(n)}$ are independent copies of $\Phi$. In
such a context stable point processes arise inevitably in various
limiting schemes similar to the central limit theorem involving
superposition of point processes. In Section
\ref{sec:general_branching_stability} we study and characterise this
class of stochastic operations. We prove that a stochastic operation
on point processes satisfies associativity and distributivity if and
only if it presents a branching structure: ``multiplying" a
point process by $t$ is equivalent to let the process evolve for time
$-\log t$ according to some general Markov branching process which may
include diffusion or jumping of the points. We
characterise branching-stable (i.e.\ stable with respect to $\bullet$)
point processes for some specific choices of $\bullet$, pointing out
possible ways to obtain a characterisation for general branching
operations. In order to do so we introduce a stochastic
operation for continuous frameworks based on continuous-state branching Markov
processes and conjecture that branching stability of point processes
and continuous-branching stability of random measures should be
related in general: branching-stable point processes are Cox processes driven by
branching-stable random measures. 

\section{Preliminaries}\label{sec:preliminaries}
In this section we fix the notation and provide the necessary facts
about branching processes and point processes that we will use.
We then address the notion of discrete
stability for random variables and point processes that we generalise
in subsequent sections.

\subsection{Branching processes refresher}
\label{sec:brachn-proc-refr}
Here we present some results from \cite[Ch.III]{AN:72},
\cite[Ch.V]{Harris:63}, \cite{Lamperti:67} and \cite{SHV:82} about
continuous branching processes that we will need. Let
$\left(Y_s\right)_{s\geq 0}$ be a $\Z_+$-valued continuous-time Markov
branching process with $Y_0=1$ almost surely, where $\Z_+$ denotes the
set of non-negative integers.  Markov branching processes describe the
evolution of the total size of a collection of particles undergoing
the following dynamic: each particle, independently of the others,
lives for an exponential time (with fixed parameter) and then it
\emph{branches}, meaning that it is replaced by a random number of
offspring particles (according to a fixed probability distribution),
which then start to evolve independently.  Such a branching process is
governed by a family of probability generating functions (p.g.f.'s)
$\mathcal{F}= (F_s)_{s \geq 0}$, where $F_s$ is the p.g.f.\ of the
integer-valued random variable $Y_s$ for every $s\geq 0$.  It is
sufficient for us here to consider the domain of $F_s$ to be $[0,1]$.
It is well known that the family $\mathcal{F}$ is a composition
semigroup:
\begin{equation}\label{semigroup_property}\tag{C1}
  F_{s+t}(\cdot)=F_s\big(F_t(\cdot)\big)\quad\forall s,t\geq 0.
\end{equation}\
Conversely, by writing relation~\eqref{semigroup_property} explicitly
in terms of a power series, it is straightforward to see that the
system of p.g.f.'s corresponding to an non-negative integer valued
random variables describes a system of particles branching
independently with the same offspring distribution, so that $Y_s$ is
the number of particles at time $s$, see also~\cite[Ch.V.5]{Harris:63}.

We require that the branching process is \emph{subcritical}, i.e.\
$\mathbb{E}[Y_s]=F'_s(1)<1$ for $s>0$. Rescaling, if necessary, the time by a
constant factor, we may assume that
\begin{equation}\label{eq:subcriticality}\tag{C2}
  \mathbb{E}[Y_s]=F'_s(1)=e^{-s}.
\end{equation}\
Finally we require the following two regularity conditions to hold:
\begin{align}
  & \lim_{s\downarrow 0}F_s(z)=F_0(z)=z    &&0\leq z\leq 1,
                                              \label{first_condition_br_proc}\tag{C3}\\
  & \lim_{s\rightarrow\infty}F_s(z)=1   &&0\leq z\leq 1.
\label{second_condition_br_proc}\tag{C4}
\end{align}\
\eqref{first_condition_br_proc} implies that the process starts with a single particle $Y_0=1$
and \eqref{second_condition_br_proc} is a consequence of the subcriticality meaning that
eventually $Y_s=0$.

A rationale behind requiring \eqref{eq:subcriticality},
\eqref{first_condition_br_proc} and \eqref{second_condition_br_proc}
will become clear later, see Remark~\ref{reason_requirements}.
Identities (\ref{semigroup_property}) and
(\ref{first_condition_br_proc}) imply the continuity and differentiability of $F_s(z)$ with
respect to $s$, \seg~\cite[Sec.III.3]{AN:72}, and thus one can define
the \emph{generator of the semigroup} $\mathcal{F}$
\begin{displaymath}
 U(z)\bydef \frac{\partial}{\partial s}\Big| _{s=0}  F_s (z) \qquad 0\leq z\leq 1.
\end{displaymath}
The function $U(\cdot)$ is continuous and it can be used to define the
\textit{A-function} relative to the branching process
\begin{equation}\label{defi_a_function}
A(z)\bydef \exp \Big[-\int_0^z\frac{dx}{U(x)}\Big]\qquad 0\leq z\leq 1,
\end{equation}\
which is a continuous strictly decreasing function with $A(0)=1$ and
$A(1)=0$, \seg~\cite[Sec.III.8]{AN:72}. From \eqref{semigroup_property} it follows that
$U(F_s(z)) =U(z) F_s'(z)$ and therefore
\begin{equation}\label{A_property_1}
  A\big(F_s(z)\big)=e^{-s}A(z)\qquad s\geq 0,\hspace{1mm} 0\leq z\leq 1.
\end{equation}
\begin{definition}\label{defi:Y_infty}
  Let $\left(Y_s\right)_{s\geq 0}$ and $\mathcal{F}= (F_s)_{s \geq 0}$
  be as above. The \emph{limiting conditional distribution} (or Yaglom
  distribution) of $Y_s$ is the weak limit of the distributions of
  $(Y_s|Y_s>0)$ when $s\rightarrow+\infty$.
  We denote by $Y_\infty$ the corresponding random variable
  and by $B(\cdot)$ its p.g.f., called the \textit{B-function} of $Y_s$.
\end{definition}
The \textit{B-function} of $Y_s$ is given by
\begin{equation}\label{AB_relationship}
  B(z)\bydef 1-A(z)=\lim_{s\rightarrow
    +\infty}\frac{F_s(z)-F_s(0)}{1-F_s(0)}, \qquad 0\leq z\leq 1.
\end{equation}\
From \eqref{A_property_1} and \eqref{AB_relationship} it follows that
\begin{equation}\label{B_property_1}
  B\big(F_s(z)\big)=1-e^{-s}+e^{-s}B(z),\qquad s\geq 0,\hspace{1mm} 0\leq z\leq 1.
\end{equation}\

Both $A$ and $B$ are continuous, strictly monotone,
and surjective functions from $[0,1]$ to $[0,1]$, thus the inverse
functions $A^{-1}$ and $B^{-1}$ exist and have the same
properties. Moreover, using \eqref{A_property_1} we obtain
  \begin{displaymath}
    \frac{d}{ds} A(F_s(0))\Big|_{s=0}=\frac{d}{ds}\Big|_{s=0} e^{-s}=1.
  \end{displaymath}
  At the same time
  \begin{displaymath}
    \frac{d}{ds}\Big|_{s=0} A(F_s(0))=A'(0)\frac{d}{ds}\Big|_{s=0}
    F_s(0) = A'(0)\frac{d}{ds}\Big|_{s=0} \P\{Y(s)=0\}.
\end{displaymath}
Since $Y_s$ is a continuous Markov branching process, every particle
branches after exponentially distributed time with a non-null
probability to die out and it follows that
\begin{displaymath}
  \frac{d}{ds}\Big|_{s=0}\P\{Y_s=0\}>0
\end{displaymath}
implying also that
\begin{equation}\label{eq:Aprim}
  A'(0)=\Bigl[\frac{d}{ds}\Big|_{s=0}\P\{Y_s=0\}\Bigr]^{-1}\in (0,+\infty).
\end{equation}

The simplest, but important for the sequel example is provided by a
\emph{pure-death process}.
\begin{example}\label{example_thinning}
  Let $\left(Y_s\right)_{s\geq 0}$ be a continuous-time pure-death
  process starting with one individual
\begin{equation*}\label{defi_pure_death_process}
  Y_s=
  \begin{cases}
    1 & \text{if }s<\tau, \\ 
    0 & \text{if } s\geq\tau,
  \end{cases}
\end{equation*}\
where $\tau$ is an exponential random variable with parameter 1. The
composition semigroup $\mathcal{F}=\big(F_s\big)_{s\geq 0}$ driving
such a process is
\begin{equation}\label{defi_thinning_semigroup}
F_s(z)=1-e^{-s}+e^{-s}z\qquad 0\leq z\leq 1.
\end{equation}\
Clearly $\mathcal{F}=\big(F_s\big)_{s\geq 0}$ satisfies
\eqref{semigroup_property}--\eqref{second_condition_br_proc}.  The
generator $U(z)$ and the $A$ and $B$-functions defined above are
\begin{equation}\label{UAB_function_thinning}
U(z)=A(z)=1-z,\qquad B(z)=z, \qquad 0\leq z\leq 1.
\end{equation}
\end{example}

Another example is the \emph{birth and death process}.
 \begin{example}\label{ex:branching_geom}
   Given two positive parameters $\lambda$ and $\mu$, assume that each
   particle disappears from the system at rate $\mu$ or it is replaced
   with two particles at rate $\lambda$ independently of the
   others. The total number of particles at each time can either grow
   or diminish by one, hence this process is also called the
   \emph{linear branching}. Its generator is given by
   \begin{equation}
     \label{eq:1}
     U(z)=\mu-(\lambda+\mu)z+\lambda z^2.
   \end{equation}
   The process is subcritical whenever $\mu>\lambda$ and in order to
   satisfy \eqref{eq:subcriticality} one needs to scale the time so
   that $\mu=\lambda+1$. This defines a one-parametric family of semigroups
   \begin{equation*}
       F_s(z)=1-\frac{e^{-s}(1-z)}{1+\lambda(1-e^{-s})(1-z)},
     \end{equation*}\
   see \cite[p.~109]{AN:72}. Conditions
   (\ref{semigroup_property}),\eqref{first_condition_br_proc} and
   (\ref{second_condition_br_proc}) also hold and the functions $A$ and $B$
   are given by
   \begin{equation*}
     A(z)=\frac{(\lambda+1)(1-z)}{1+\lambda(1-z)},\qquad B(z)=\frac{z}{1+\lambda(1-z)},
   \end{equation*}
   for $z$ in $[0,1]$. Thus $B$ describes the p.g.f.\ of a (shifted) Geometric
   distribution with parameter $(1+\lambda)^{-1}$.
 \end{example}

\subsection{Point processes refresher}
\label{sec:point-proc-refr}
We now pass to the necessary definitions related to point processes.
The details can be found, for instance, in~\cite{DVJ1}, \cite{DVJ2}
and \cite{Kal:83}. A \emph{random measure} on a \emph{phase space}
$\sX$ which we assume to be a locally compact second countable
Hausdorff space, is a measurable mapping $\xi$ from some
probability space $(\Omega,\sF,\P)$ into the measurable space
$(\sM,\sB(\sM))$, where $\sM$ denote the set of all Radon measures on
the Borel \salg $\sB(\sX)$ of subsets of $\sX$ and $\sB(\sM)$ is the
minimal \salg that makes the mappings $\mu\mapsto \mu(B)$,\
$\mu\in\sM$ measurable for all $B\in \sB(\sX)$.

The distribution of a random measure is characterised by the
\emph{Laplace functional} $L_\xi[u]$ which is defined for the class
$\bmp$ of non-negative bounded functions $u$ with bounded support by
means of
\begin{equation}
  \label{eq:2}
  L_\xi[u]=\E \exp\bigl\{ - \langle u,\xi\rangle \bigr\},\quad u\in\bmp.
\end{equation}\
Here and below $\langle u, \mu\rangle$ stands for the integral $\int
u(x)\,\mu(dx)$ over the the whole $\sX$ unless
specified otherwise. 

A \emph{point process} (p.p.) $\Phi$ is a random \emph{counting}
measure, i.e.\ a random measure that with probability one takes values
in the set $\sN$ of all boundedly finite counting measures on
$\sB(\sX)$. The corresponding \salg $\sB(\sN)$ is the restriction
of $\sB(\sM)$ onto $\sN$. The \emph{support} of
$\phi\in\sN$ is the set $\supp \phi\bydef\{x\in\sX:\
\phi(\{x\})>0\}$. A point process is called \emph{simple} (or without
multiple points) if $\P\Phi^{-1}\{\phi\in\sN:\ \phi(\{x\})\leq 1\
\forall x\in\sX\}=1$. The distribution of a point process $\Phi$ can
be characterised by the probability generating functional (p.g.fl.)
$G_\Phi[h]$ defined for functions $h$ such that $0<h(x)\leq 1$ for all
$x\in\sX$ and such that the set $\{x\in\sX:\ h(x)\neq 1\}$ is compact.
We denote the class of such functions by $\sV(\sX)$. Then
\begin{equation*}
  G_\Phi[h]=L_\Phi[-\log h]=
  \E \exp\bigl\{ \langle \log h,\Phi\rangle\bigr\},\quad h\in\sV(\sX).
\end{equation*}\
For a simple p.p.\ $\Phi$, this expression simplifies to
\begin{displaymath}
  G_\Phi[h]=\E \prod_{x_i\in\supp \Phi} h(x_i),\quad h\in\sV(\sX).
\end{displaymath}

A \emph{Poisson point process} with \emph{intensity measure} $\Lambda$ is
the p.p.\ $\Pi$ having the p.g.fl.
\begin{equation*}
  G_\Pi[h]=\exp\bigl\{ -\langle 1- h,\Lambda\rangle\bigr\},\quad h\in\sV(\sX).
\end{equation*}\
It is characterised by the following property: given a family of disjoint
sets $B_i\in\sB(\sX)$,\ $i=1\thru n$, the counts $\Pi(B_1)\thru
\Pi(B_n)$ are mutually independent Poisson
$\Po(\Lambda(B_i))$ distributed random variables for $i=1\thru n$.

Given a random measure $\xi$, a \emph{Cox process} with
\emph{parameter measure} $\xi$ is the point process with the \pgfl
\begin{equation}\label{eq:pgfl_cox}
  G_\Phi[h]=\E\exp\bigl\{ -\langle 1- h,\xi\rangle\bigr\},\quad h\in\sV(\sX).
\end{equation}\
It is called \emph{doubly-stochastic}, since it can be constructed by first
taking a realisation $\xi(\omega)$ of the parameter measure and then taking
a realisation of a Poisson p.p.\ with intensity measure $\xi(\omega)$.

Consider a family of point processes $(\Psi_y)_{y\in\sY}$ on $\sX$
indexed by the elements of a locally compact and second countable
Hausdorff space $\sY$ which may or may not be
$\sX$ itself. Such a family is called a \emph{measurable family} if
$\P_y(A)\bydef \P(\Psi_y\in A)$ is a $\sB(\sY)$-measurable function
of $y$ for all $A\in\sB(\sN)$. 

Given a point process $\Xi$ on $\sY$ and a measurable family of point
processes $(\Psi_y)_{y\in\sY}$ on $\sX$, the \emph{cluster process} is
the following random measure:
\begin{equation}\label{eq:cldef}
  \Phi(\cdot)=\int_\sY \Psi_y(\cdot)\,\Xi(dy)
\end{equation}\
The p.p.\ $\Xi$ is then called the \emph{center} process and $\Psi_y,\
{y\in\sY}$ are called the \emph{component processes} or
\emph{clusters}.  The commonest model is when the clusters in
\eqref{eq:cldef} are independent for different $y_i\in\supp\Xi$ given
a realisation of $\Xi$. In this case, if $G_\Xi[h]$ is the p.g.fl.\ of
the center process and $G_\Psi[h|y]$ are the p.g.fl.'s of $\Psi_y$,
$y\in\sY$, then the p.g.fl.\ of the corresponding cluster process
\eqref{eq:cldef} is given by the composition
\begin{equation}\label{eq:clg}
  G_\Phi[h]=G_\Xi\bigl[G_\Psi[h|\,\cdot\,]\bigr].
\end{equation}\

\subsection{Stability for discrete random variables}\label{sec:das0}
Let $X$ be a $\mathbb{Z}_+$-valued random
variable. As in~\cite{ste:har79}, we define an operation of discrete
multiplication $\circ$ by a number $t\in[0,1]$
\begin{equation}\label{eq:dthin}
  t\circ X\stackrel{\mathcal{D}}=\sum_{i=1}^{X}Z^{(i)},
\end{equation}\
where $\{Z^{(i)}\}_{i\in\mathbb{N}}$ are independent and identically
distributed (i.i.d.)\ random variables with Bernoulli distribution
$\Bin(1,t)$. A random variable $X$ (or its distribution) is then called
\emph{discrete $\alpha$-stable} (notation: $\das$) if
\begin{equation}\label{das0}
  t^{1/\alpha}\circ X'+(1-t)^{1/\alpha}\circ X''
  \deq X\qquad\text{for all}\ t\in [0,1],
\end{equation}\
where $X',X''$ are independent distributional copies of $X$. 

Letting each point $i$ evolve as a pure-death process
$Y^{(i)}$ independently of the others (see
Example~\ref{example_thinning}), after time $-\ln t$, $t\in(0,1]$, the
number of surviving points will be distributed as~\eqref{eq:dthin}. So
alternatively,
\begin{displaymath}
  t\circ X\stackrel{\mathcal{D}}=\sum_{i=1}^{X}Y^{(i)}_{-\ln t}\,.
\end{displaymath}
Replacing here the pure-death with a general branching process allowed
the authors of~\cite{SHV:82} to define a more general \emph{branching}
operation and the corresponding \emph{$\sF$-stable} non-negative integer
random variables as follows.  Let $\{Y^{(i)}\}_{i\in\mathbb{N}}$ be
a sequence of i.i.d.\ continuous-time Markov branching processes
driven by a semigroup $\mathcal{F}= (F_s)_{s \geq 0}$ satisfying the
conditions (C1)-(C4) in the previous section. Given $t\in (0,1]$ and a
$\mathbb{Z}_+$-valued random variable $X$ (independent of
$\{Y^{(i)}\}_{i\in\mathbb{N}}$) define
\begin{equation}\label{defi_of_F_operation}
  t\circ_\mathcal{F}X\bydef\sum_{i=1}^X Y^{(i)}_{-\ln t}\,.
\end{equation}\
The notion of $\mathcal{F}$-stability is then defined in an analogous
way to discrete stability:
\begin{definition}\label{defi:f_stability_rv}
  A $\mathbb{Z}_+$-valued random variable $X$ (or its distribution) is called
  $\mathcal{F}$\emph{-stable with exponent} $\alpha$ if
  \begin{equation}\label{F_stable_rv_no_pgf}
    t^{1/\alpha}\circ_\mathcal{F}X'+(1-t)^{1/\alpha}\circ_\mathcal{F}X''
    \stackrel{\mathcal{D}}=X\qquad\forall t\in[0,1],
  \end{equation}\
  where $X'$ and $X''$ are independent copies of $X$.
\end{definition}
In terms of the p.g.f.\ $G_X(z)$ of $X$, \eqref{F_stable_rv_no_pgf} is equivalent to
\begin{equation*}
  G_X(z)=G_X\big(F_{-\alpha^{-1}\ln t}(z)\big)
  \cdot G_X\big(F_{-\alpha^{-1}\ln(1-t)}(z)\big)\qquad0\leq z\leq1.
\end{equation*}\

Let $G_{t\circ_\mathcal{F}X}(z)$ denote the p.g.f.\ of
$t\circ_\mathcal{F}X$. By independence of
$\{Y^{(i)}(\cdot)\}_{i\in\mathbb{N}}$ and $X$, (\ref{defi_of_F_operation})
is equivalent to
\begin{equation}\label{pgf_of_F_operation}
G_{t\circ_\mathcal{F}X}(z)=G_X\big(F_{-\ln t}(z)\big)\qquad0\leq z\leq1.
\end{equation}\
It is easy to verify that \eqref{semigroup_property} and
\eqref{pgf_of_F_operation} make the branching
operation $\circ_\mathcal{F}$ associative, commutative and
distributive with respect to the sum of random variables, i.e.\ for
all $t,t_1,t_2\in [0,1]$ and $X$ independent of $X'$
\begin{gather}
  t_1\circ_\mathcal{F}(t_2\circ_\mathcal{F}X)\stackrel{\mathcal{D}}
  =(t_1t_2)\circ_\mathcal{F}X\stackrel{\mathcal{D}}
  =t_2\circ_\mathcal{F}(t_1\circ_\mathcal{F}X), \label{eq:fass}\\
  t\circ_\mathcal{F}(X+X')\stackrel{\mathcal{D}}
  =t\circ_\mathcal{F}X + t\circ_\mathcal{F}X'.\label{eq:fcom}
\end{gather}\

\begin{remark}\label{reason_requirements} As shown in 
\cite[Section V.8, equations (8.6)-(8.8)]{SVH:04}, conditions \eqref{eq:subcriticality},
  \eqref{first_condition_br_proc} and \eqref{second_condition_br_proc}
  guarantee that $\circ_\mathcal{F}$ has some ``multiplication-like''
  properties. In particular \eqref{first_condition_br_proc} and
  \eqref{second_condition_br_proc} imply respectively that
  $\lim_{t\uparrow1}t\circ_\mathcal{F}X\stackrel{\mathcal{D}}
  =1\circ_\mathcal{F}X\stackrel{\mathcal{D}}=X$ and
  $\lim_{t\downarrow0}t\circ_\mathcal{F}X\stackrel{\mathcal{D}}=0$.
  Furthermore, \eqref{eq:subcriticality} implies that, in case the
  expectation of $X$ is finite, $\E[t\circ_\mathcal{F}X]=t\E X$. 
\end{remark}

The following theorem gives a characterisation of $\mathcal{F}$-stable
distributions on $\Z_+$, see \cite[Theorem~7.1]{SHV:82} and \cite[Theorem~V.8.6]{SVH:04}:
\begin{theorem}\label{th:charFstabrv}
  Let $X$ be a $\mathbb{Z}_+$-valued random variable and $G_X(z)$ its
  p.g.f., then $X$ is $\mathcal{F}$-stable with exponent $\alpha$ if
  and only if $0<\alpha\leq 1$ and
  \begin{displaymath}
    G_X(z)=\exp\big\{-cA(z)^{\alpha}\big\}\qquad0\leq z\leq1,
  \end{displaymath}
  where A is the A-function \eqref{defi_a_function} associated to the
  branching process driven by the semigroup $\mathcal{F}$ and $c>0$.
  In particular, $X$ is $\das$ if and only if 
  \begin{displaymath}
    G_X(z)=\exp\big\{-c(1-z)^{\alpha}\big\}\qquad0\leq z\leq1
  \end{displaymath}
  for some $0<\alpha\leq 1$ and $c>0$, see \cite[Theorem~3.2]{ste:har79}.
\end{theorem}

\subsection{Thinning stable point processes}\label{sec:das_pp}
It was noted in Section~\ref{sec:das0} that the operation of discrete
multiplication of an integer random variable $t\circ$ with $t\in[0,1]$
may be thought of as an independent thinning when the random
variable is represented as a collection of points and each point is
retained with probability $t$ and removed with the complementary
probability. Thus the thinning operation generalises the discrete
multiplication to general point processes. The corresponding
\emph{thinning-stable} or \emph{discrete $\alpha$-stable} point processes
(notation: \das) satisfy~\eqref{das} and are
exactly the ones which appear as the limit in thinning-superposition
schemes, see~\cite[Ch.8.3]{Kal:83}. The full characterisation of these
processes is given in~\cite{DavMolZuy:11}. Thinning
stable processes exist only for $\alpha\in(0,1]$, and the case $\alpha=1$
corresponds to the Poisson processes.

To be more specific, we need some further definitions. First we need
a way to consistently normalize both finite and infinite measures. Let
$B_1,B_2,\ldots$ be a fixed countable base of the topology on $\sX$ that
consists of relatively compact sets.  Append $B_0=\sX$ to this base.
For each non-null $\mu\in\sM$ consider the sequence of its values
$(\mu(B_0),\mu(B_1),\mu(B_2),\ldots)$ possibly starting with infinity,
but otherwise finite. Let $i(\mu)$ be the smallest non-negative
integer $i$ for which $0<\mu(B_i)<\infty$, in particular, $i(\mu)=0$ if $\mu$ is
a finite measure. Define
\begin{equation*}
  \SS=\{\mu\in\sM:\; \mu(B_{i(\mu)})=1\}.
\end{equation*}\
It can be shown (see \cite{DavMolZuy:11}) that $\SS$ is
$\sB(\sM)$-measurable and that $\SS\cap \{\mu:\
\mu(\sX)<\infty\}=\M_1$ is the family of all probability measures on
$\sX$. Furthermore, every
$\mu\in\sM\setminus\{0\}$ can be uniquely associated with the pair
$(\hat{\mu},\mu(B_{i(\mu)}))\in \SS\times \R_+$, where $\hat{\mu}$
is defined as $\frac{\mu}{\mu(B_{i(\mu)})}$\,, and 
$\mu=\mu(B_{i(\mu)})\hat{\mu}$ is the \emph{polar representation} of $\mu$.

A locally finite random measure $\xi$ is called \emph{strictly stable
  with exponent} $\alpha$ or \sas\ if it satisfies identity
\eqref{definition}. It is deterministic in the case $\alpha=1$ and in
the case $\alpha\in(0,1)$ its Laplace functional is given by
 \begin{equation}
   \label{eq:lf}
   L_\xi[h]=\exp\Bigl\{-\int_{\sM\setminus\{0\}}
   (1-e^{-\langle h,\mu\rangle}) \Lambda(d\mu)\Bigr\}\,,
   \quad h\in\bmp\,,
 \end{equation}
 where $\Lambda$ is a \emph{L\'evy measure}, i.e.\ a Radon measure on
 $\sM\setminus\{0\}$ such that
 \begin{equation}
   \label{eq:lmreg}
   \int_{\sM\setminus\{0\}}(1-e^{-\langle h,\mu\rangle})
   \Lambda(d\mu)<\infty
 \end{equation}\
 for all $h\in\bmp$. Such $\Lambda$ is homogeneous of order $-\alpha$, i.e.\
 $\Lambda(tA)=t^{-\alpha}\Lambda(A)$ for all measurable
 $A\subset\sM\setminus\{0\}$ and $t>0$, see \cite[Th.~2]{DavMolZuy:11}.

Introduce a \emph{spectral measure} $\sigma$ supported by $\SS$ by setting
\begin{displaymath}
 \sigma (A)=\Gamma(1-\alpha)\,\Lambda(\{t\mu:\; \mu\in A,\; t\geq1\})
\end{displaymath}
for all measurable $A\subset\SS$, where $\Gamma$ is the Euler's
Gamma-function. Integrating out the radial component in \eqref{eq:lf}
leads to the following alternative representation \cite[Th.~3]{DavMolZuy:11}:
  \begin{equation}
    \label{eq:sd}
    L_\xi[u]=\exp\Bigl\{-\int_{\SS}
    \langle u,\mu\rangle^\alpha \sigma(d\mu)\Bigr\}\,,
    \quad u\in\bmp
  \end{equation}\
for some \emph{spectral measure} $\sigma$ supported by $\SS$ which satisfies
\begin{equation}
  \label{eq:s-fin}
  \int_{\SS} \mu(B)^\alpha \sigma(d\mu)<\infty
\end{equation}\
for all relatively compact subsets $B$ of $\sX$. The latter is a
consequence of \eqref{eq:lmreg} and representation \eqref{eq:sd}
is unique.

The importance of \sas\ random measures is explained by the fact that
any \das\ point process $\Phi$ is exactly a Cox processes driven by a \sas\
parameter measure $\xi$: its \pgfl\ has the form
\begin{equation}
  \label{eq:pg-pp-spec}
  G_{\Phi}[h]=L_\xi[1-h]=\exp\Bigl\{-\int_{\SS}
  \langle 1-h,\mu\rangle^\alpha \sigma(d\mu)\Bigr\}\,,
  \quad h\in\sV(\sX)
\end{equation}\
for some locally finite spectral measure $\sigma$ on $\SS$ that
satisfies (\ref{eq:s-fin}), see \cite[Th.~15 and Cor.~16]{DavMolZuy:11}.

In the case when $\sigma$ charges only probability measures $\M_1$,
the corresponding \das\ p.p.'s are cluster processes. Recall that a
positive integer random variable $\eta$ has \emph{Sibuya} $\Sib(\alpha)$ distribution with
parameter $\alpha$, if its \pgf\ is given by
\begin{displaymath}
  \E z^\eta=1-(1-z)^\alpha\,,\quad z\in(0,1]\,. 
\end{displaymath}
It corresponds to the number of
trials to get the first success in a series of Bernoulli trials with
probability of success in the $k$th trial being $\alpha/k$.
\begin{definition}(See \cite[Def.23]{DavMolZuy:11})
  Let $\mu$ be a probability measure on $\sX$. A point process
  $\Upsilon$ on $\sX$ defined by the \pgfl
  \begin{equation}
    \label{eq:sibpgfl}
    G_\Upsilon[h]=G_{\Upsilon(\mu)}[h]=1-\langle
    1-h,\mu\rangle^\alpha,  \quad h\in\sV(\sX),
  \end{equation}\
  is called a \emph{Sibuya point process} with \emph{exponent}
  $\alpha$ and \emph{parameter measure} $\mu$. Its distribution is
  denoted by $\Sib(\alpha,\mu)$.
\end{definition}
A Sibuya process $\Upsilon\sim\Sib(\alpha,\mu)$ is a.s.\ finite, the
total number of its points $\Upsilon(\sX)$ follows $\Sib(\alpha)$
distribution and, given the total number of points, these points are
independently identically distributed in $\sX$ according to $\mu$.
\begin{theorem}[Th.~24 \cite{DavMolZuy:11}] 
    \label{th:pois-sib}
  A $\das$ point process $\Phi$ with a spectral measure $\sigma$
  supported by $\M_1$ can be represented as a cluster process with
  Poisson centre process on $\M_1$ driven by intensity measure
  $\sigma$ and component processes being Sibuya processes
  $\Sib(\alpha,\mu),\ \mu\in\M_1$. Its \pgfl is given by
  \begin{equation*}
    G_\Phi[h]=\exp\Bigl\{\int_{\M_1}
    (G_{\Upsilon(\mu)}[h]-1)\, \sigma(d\mu)\Bigr\}, \quad h\in\sV(\sX),
  \end{equation*}\
  with $G_{\Upsilon(\mu)}[h]$ as in~\eqref{eq:sibpgfl}. 

\end{theorem}

\section{$\mathcal{F}$-stability for point processes}\label{f_stab_pp}

We have seen in the previous section that the discrete multiplication
operation on integer random variables generalises to the thinning
operation on points processes. In a similar fashion, we can extend the
branching operation $\circ_\mathcal{F}$ to point processes too.

Let $\{Y_s\}_{s\geq 0}$ be a continuous-time Markov branching process
driven by a semigroup $\mathcal{F}=(F_s)_{s\geq 0}$ satisfying
conditions
\eqref{semigroup_property}-\eqref{second_condition_br_proc}. Intuitively,
given a point process $\Phi$ and $t\in(0,1]$, $t\circ_\mathcal{F}\Phi$
is the cluster point process obtained from $\Phi$ by replacing every
point with $Y_{-\ln t}$ points located in the same position (using an
independent copy of $Y_{-\ln t}$ for each point including the ones in
the same position). In this sense, the resulting process is a cluster
process. To proceed formally, we first define its component
processes.
\begin{definition}\label{defi:Y_x} Given a
  $\mathbb{Z}_+$-valued random variable $Z$ and
  $x\in\mathcal{X}$, we denote by $Z_x$ the point process having $Z$
  points in $x$ and no points in $\mathcal{X}\backslash\{x\}$.
  Equivalently $Z_x$ is the point process with \pgfl
  $G_{Z_x}[h]=F\big(h(x)\big)$ for each
  $h\in\mathcal{V}(\mathcal{X})$, where $F(z)$ is the \pgf of $Z$.
\end{definition}

We can now define the operation $\circ_\mathcal{F}$ for point
processes. 

\begin{definition} Let $\Phi$ be a point process and
  $\{Y_s\}_{s\geq 0}$ be a continuous-time Markov branching process
  driven by a semigroup $\mathcal{F}=(F_s)_{s\geq 0}$ satisfying
  conditions
  \eqref{semigroup_property}-\eqref{second_condition_br_proc}. For
  each $t\in(0,1]$, $t\circ_\mathcal{F}\Phi$ is the (independent)
  cluster point process with center process $\Phi$ and clusters
  $\big\{(Y_{-\ln t})_x,\ x\in \mathcal{X}\big\}$. 
\end{definition}

Equivalently $t\circ_\mathcal{F}\Phi$ can be defined as the point
process with \pgfl given by
\begin{equation*}
  G_{t\circ_\mathcal{F}\Phi}[h]=G_\Phi[F_{-\ln t}(h)],
\end{equation*}\ 
where $G_\Phi$ is the \pgfl of $\Phi$. Note that
$t\circ_\mathcal{F}\Phi$ does not need to be simple (i.e.\ it can have
multiple points), even if $\Phi$ is. We define the
$\mathcal{F}$-stability for point processes as follows. 
\begin{definition} A
  p.p.\ $\Phi$ is $\mathcal{F}$-stable with exponent $\alpha$
  ($\alpha$-stable with respect to $\circ_\mathcal{F}$) if
  \begin{equation}\label{defi_fstable_pp}
    t^{1/\alpha}\circ_\mathcal{F}\Phi'+(1-t)^{1/\alpha}\circ_\mathcal{F}\Phi''
    \deq \Phi\qquad\forall
    t\in(0,1],
  \end{equation}\
where  $\Phi'$ and $\Phi''$ are independent copies of $\Phi$.
\end{definition}
Equivalently, (\ref{defi_fstable_pp}) can be rewritten in terms of
p.g.fl.'s\ as follows:
\begin{equation*}
  G_\Phi[h]=G_\Phi\big[F_{-\ln t/\alpha}(h)\big]\,
  G_\Phi\big[F_{-\ln(1-t)/\alpha}(h)\big]\qquad\forall
  t\in(0,1],\forall h\in\mathcal{V}(\mathcal{X}). 
\end{equation*}\

\begin{remark}\label{parallel_with_thinning_pp}
  The branching operation $\circ_\mathcal{F}$ induced by the
  pure-death process of Example~\ref{example_thinning} corresponds to
  the thinning operation. 
Therefore \das\ point processes can be seen as a special case of
$\mathcal{F}$-stable point processes.
\end{remark}

An $\mathcal{F}$-stable point process $\Phi$ is necessarily infinitely divisible.
Indeed, iterating \eqref{defi_fstable_pp} $m-1$ times we obtain
\begin{equation}\label{F_stable_natural}
  m^{-1/\alpha}\circ_\mathcal{F}\Phi^{(1)}+...+m^{-1/\alpha}
  \circ_\mathcal{F}\Phi^{(m)}\deq \Phi,
\end{equation}\
where $\Phi^{(1)},...,\Phi^{(m)}$ are independent copies of $\Phi$. 

A characterisation of $\mathcal{F}$-stable point processes is given in
the following theorem which generalises \cite[Th.15]{DavMolZuy:11} and
which proof we largely follow here.
\begin{theorem}\label{teo15}
  A functional $G_\Phi [\cdot]$ is the \pgfl of an
  $\mathcal{F}$-stable point process $\Phi$ with exponent of stability
  $\alpha$ if and only if $0<\alpha\leq 1$ and there exists a \sas\ random measure
  $\xi$ such that
\begin{equation}\label{pgfl_fstable}
  G_\Phi [h]=L_{\xi}\big[A(h)\big]
  =L_{\xi}\big[1-B(h)\big]\qquad\forall h\in\sV,
\end{equation}\
where $A(z)$ and $B(z)$ are the $A$-function and $B$-function of the
branching process driven by $\mathcal{F}$.
\end{theorem}
\begin{proof}
  \textit{Sufficiency:} Suppose \eqref{pgfl_fstable} holds. As it was
  shown in \cite{DavMolZuy:11}, \sas\ random measures exist only for
  $0<\alpha\leq 1$, $\alpha=1$ corresponding to non-random measures.
  Next, by \eqref{eq:2} and \eqref{eq:pgfl_cox}, $L_{\xi}\big[1-h\big]$ as a
  functional of $h$ is the p.g.fl.\ of a Cox point process with intensity $\xi$ and $B(z)$
  is the p.g.f.\ of the limiting conditional distribution of the
  branching process driven by $\mathcal{F}$. Therefore, by \eqref{eq:clg}, the functional
  $G_\Phi [h]=L_{\xi}\big[1-B(h)\big]$ is the \pgfl of a cluster
  process, say $\Phi$. We need to prove that $\Phi$ is
  $\mathcal{F}$-stable with exponent $\alpha$. Given $t\in(0,1]$ and
  $h\in\mathcal{V}(\mathcal{X})$ it holds that
  \begin{multline*}
    G_\Phi\big[F_{-\ln t/\alpha}(h)\big]\,
    G_\Phi\big[F_{-\ln(1-t)/\alpha}(h)\big]
=\\=
L_\xi\Big[A\big(F_{-\ln
      t/\alpha}(h)\big)\Big]\,
    L_\xi\Big[A\big(F_{-\ln(1-t)/\alpha}(h)\big)\Big]
\stackrel{\eqref{A_property_1}}{=}
    L_\xi\big[t^{1/\alpha}A(h)\big]\,  L_\xi\big[(1-t)^{1/\alpha}A(h)\big].
  \end{multline*}\
Since $\xi$ is \sas, it satisfies \eqref{definition} and thus 
\begin{equation*}
  L_\xi\big[t^{1/\alpha}A(h)\big]\,  L_\xi\big[(1-t)^{1/\alpha}A(h)\big]
  =
L_\xi\big[A(h)\big].
\end{equation*}\
Therefore
\begin{displaymath}
  G_\Phi\big[F_{-\ln t/\alpha}(h)\big]\,
  G_\Phi\big[F_{-\ln(1-t)/\alpha}(h)\big]
  =G_\Phi[h]
\end{displaymath}
for any $h$ in $\mathcal{V}(\mathcal{X})$, meaning that $\Phi$ is $\mathcal{F}$-stable with exponent $\alpha$.

\bigskip
\noindent
\textit{Necessity:} Suppose that $\Phi$ is $\mathcal{F}$-stable with
exponent $\alpha$. Writing \eqref{defi_fstable_pp} for the values of
the measures on a particular compact set $B\in\sB(\sX)$, we see that
$\Phi(B)$ is an $\mathcal{F}$-stable random variable with exponent
$\alpha$. Thus by Theorem~\ref{th:charFstabrv} we have
$0<\alpha\leq1$. Now we are going to prove that $G_{\Phi}[A^{-1}(u)]$,
as a functional of $u$, is the Laplace functional of a \sas\ random
measure.  While a Laplace functional should be defined on all
(bounded) functions with compact support, the expression
$G_{\Phi}[A^{-1}(u)]$ is well defined just for functions with values
on $[0,1]$ because $A^{-1}:[0,1]\rightarrow[0,1]$. To overcome this
difficulty we employ \eqref{F_stable_natural} which can be written as
\begin{displaymath}
  G_\Phi[h]=\big(G_\Phi[F_{\alpha^{-1}\ln m}(h)]\big)^m\qquad\forall h\in\mathcal{V}(\mathcal{X}),
\end{displaymath}
and define
\begin{equation}
  \label{eq:ml}
  L[u]=\Big(G_\Phi\big[F_{\alpha^{-1}\ln
    m}\big(A^{-1}(u)\big)\big]\Big)^m\stackrel
  {\eqref{A_property_1}}
  =\Big(G_\Phi[A^{-1}(m^{-1/\alpha}u)]\Big)^m\,
  \qquad u\in \bmp,
\end{equation}\
for any $m\geq1$ such that $m^{-1/\alpha}u<1$.
Note that the right-hand side of \eqref{eq:ml} does not depend on $m$.
Moreover, given $m^{-1/\alpha}u<1$, the function
$A^{-1}(m^{-1/\alpha}u)$ does take values in $[0,1]$ and equals $1$
outside of a compact set, implying that $A^{-1}(m^{-1/\alpha}u)\in\mathcal{V}(\mathcal{X}) $.
Therefore $L[u]$ in \eqref{eq:ml} is well-defined.
Since (\ref{eq:ml}) holds for all $m$, it is
possible to pass to the limit as $m\to\infty$ to see that
\begin{equation}\label{eq:L_defi_lim}
 L[u]=\exp\Big\{-\lim_{m\to\infty}m(1-G_\Phi[A^{-1}(m^{-1/\alpha}u)])\Big\}
\end{equation}
We now need the following fact:
\begin{equation}
\lim_{m\to\infty} m(1-G_\Phi[A^{-1}(m^{-1/\alpha}u)])
 =
\lim_{m\to\infty} m(1-G_\Phi[e^{(A^{-1})(0)\,m^{-1/\alpha}u}])
\,. \label{eq:eqlim}
\end{equation}\
Since $A^{-1}$ is continuous, strictly decreasing and differentiable
in 0 with $A^{-1}(0)=1$ and $(A^{-1})'(0)< 0$ (see Section
\ref{sec:brachn-proc-refr}), it follows that for any constant
$\epsilon>0$ there exists $M(\epsilon,u)>0$ such that 
\begin{equation}\label{eq:exp_ineq}
A^{-1}(m^{-1/\alpha}u(1+\epsilon))\leq 
e^{(A^{-1})'(0)m^{-1/\alpha}u}\leq
A^{-1}(m^{-1/\alpha}u(1-\epsilon))
\qquad \forall\; m\geq M(\epsilon,u).
\end{equation}\
From \eqref{eq:L_defi_lim}, \eqref{eq:exp_ineq} and the monotonicity of $G_{\Phi}$ we can deduce
\begin{equation}\label{eq:double_control}
  L[(1+\epsilon)u]\leq
\exp\left\{-\lim_{m\rightarrow\infty} m(1-G_\Phi[e^{(A^{-1})'(0)m^{-1/\alpha}u}])\right\}
\leq L[(1-\epsilon)u]\,.
\end{equation}\
Note that $L$ is continuous because of \eqref{eq:ml} and the continuity of $G_{\Phi}$.
Therefore taking the limit for $\epsilon$ going to 0 in \eqref{eq:double_control} and using \eqref{eq:L_defi_lim} we obtain \eqref{eq:eqlim}. 

From \eqref{eq:L_defi_lim} and Schoenberg
theorem \cite[Theorem 3.2.2]{BCR} it follows that $L[u]$ is positive definite if
$\lim_{m\rightarrow\infty} m(1-G_\Phi[1-B^{-1}(m^{-1/\alpha}u)])$
is negative definite, i.e.\ by \eqref{eq:eqlim} if
\begin{equation}\label{eq:neg_def}
  \sum_{i,j=1}^n c_ic_j \lim_{m\to\infty} m(1-G_\Phi[e^{(A^{-1})(0)\,m^{-1/\alpha}(u_i+u_j)}])\leq 0,
\end{equation}
for all $n\geq2$, for any $u_1,\dots,u_n\in \bmp$ and
for any $c_1,\dots,c_n$ with $\sum c_i=0$.
If we set $v_i=e^{(A^{-1})(0)\,m^{-1/\alpha}u_i}$, then \eqref{eq:neg_def} is equivalent to 
$\sum_{i,j=1}^n c_ic_j \lim_{m\to\infty} G_\Phi[v_iv_j]\geq 0$,
which follows from the positive definiteness of $G_\Phi$.
Thus, by the Bochner theorem \cite[Theorem 4.2.9]{BCR},
the function $L[\sum_{i=1}^k t_ih_i]$ of $t_1,\dots,t_k\geq0$ is the
Laplace transform of a random vector. Moreover $L[\textbf{0}]=1$,
where $\textbf{0}$ is the null function on $\mathcal{X}$. Finally from
\eqref{eq:ml} and the continuity of the p.g.fl.\ $G_\Phi$ it follows
that given $\{f_n\}_{n\in\mathbb{N}}\subset \bmp$,
$f_n\uparrow f\in \bmp$ we have $L[f_n]\rightarrow L[f]$ as
$n\to\infty$. Therefore we can use Theorem 9.4.II in \cite{DVJ2} to
obtain that $L$ is the Laplace functional of a random measure $\xi$.

In order to prove that $\xi$ is \sas, let $u\in \bmp$ and take an integer $m\geq (\sup
u)^\alpha$ and denote by $\hat{u}=m^{-1/\alpha}u\leq 1$. By
\eqref{eq:ml}, for any given $t\in(0,1]$ we have
\begin{multline*}
  L_\xi [u]=G^m_{\Phi}[A^{-1}(\hat{u})]
\stackrel{\eqref{defi_fstable_pp}}=
G^m_\Phi\big[F_{-\ln
    t/\alpha}(A^{-1}(\hat{u}))\big]\,
  G^m_\Phi\big[F_{-\ln(1-t)/\alpha}(A^{-1}(\hat{u}))\big]
\stackrel{\eqref{A_property_1}}=\\
  G^m_\Phi\big[A^{-1}(t^{1/\alpha}\hat{u})\big]\,
  G^m_\Phi\big[A^{-1}((1-t)^{1/\alpha}\hat{u})\big]=L_\xi [t^{1/\alpha}u]\,
  L_\xi [(1-t)^{1/\alpha}u],
\end{multline*}\
which implies that $\xi$ is St$\alpha$S. 
\end{proof}
\begin{corollary}\label{fas_as_das_cluster}
  A p.p.\ $\Phi$ on $\mathcal{X}$ is $\mathcal{F}$-stable with
  exponent $\alpha$ if and only if it is a cluster process with
  \das\ centre process $\Psi$ on $\mathcal{X}$ and component
  processes $\big\{(Y_\infty)_x,\ x\in \mathcal{X}\big\}$
  (see Definitions \ref{defi:Y_infty} and \ref{defi:Y_x}).
\end{corollary}
\begin{proof}
  From Theorem \ref{teo15} and \eqref{AB_relationship} it follows that
  $\Phi$ is $\mathcal{F}$-stable if and only if its p.g.fl.\ satisfies
  $G_\Phi [h]=L_{\xi}\big[1-B(h)\big],$ where $B(\cdot)$ is the
  p.g.f.\ of $Y_\infty$, and $\xi$ is a St$\alpha$S random
  measure.
  By \eqref{eq:pg-pp-spec} there is a \das\ point process
  $\Psi$ with $G_\Psi [h]=L_{\xi}\big[1-h\big].$
  We obtain that $G_\Phi [h]=G_\Psi\big[B(h)\big]$.
  The result follows from the form \eqref{eq:clg} of
  the \pgfl of a cluster process.
\end{proof}
Corollary \ref{fas_as_das_cluster} clarifies the relationship between
$\mathcal{F}$-stable and D$\alpha$S point processes:
$\mathcal{F}$-stable p.p.'s are an extension of D$\alpha$S p.p.'s, where
every point is given an additional multiplicity according to
independent copies of $Y_\infty$ (the latter is fixed by
$\mathcal{F}$). Note that when the branching operation is thinning,
the random variable $Y_\infty$ is identically 1 (that stems from
\eqref{UAB_function_thinning}) and the $\mathcal{F}$-stable p.p.\ is the
D$\alpha$S centre process itself.

\begin{corollary}\label{coroll_spectral_fas}
 A p.p.\ $\Phi$ is $\mathcal{F}$-stable with
  exponent $0<\alpha\leq 1$ if and only if its \pgfl\ can be written as
\begin{equation}\label{spectral_representation_fstable}
  G_\Phi [u]=\exp\big\{-\int_\mathbb{S}\langle
  1-B(u),\mu\rangle^\alpha \sigma(d\mu)\big\},
\end{equation}\
where $\sigma$ is a locally finite spectral measure on $\mathbb{S}$
satisfying \eqref{eq:s-fin}.
\end{corollary}
\begin{proof}
  If $\Phi$ is an $\mathcal{F}$-stable point process with stability
  exponent $\alpha$, then by Theorem~\ref{teo15} there exist a
  St$\alpha$S random measure $\xi$ such that
  \begin{displaymath}
    G_\Phi [h]=L_{\xi}\big[A(h)\big]\qquad h\in\mathcal{V}(\mathcal{X}).
\end{displaymath}
Thus (\ref{spectral_representation_fstable}) follows from spectral
representation \eqref{eq:sd}. Conversely, if we
have a locally finite spectral measure $\sigma$ on $\mathbb{S}$
satisfying \eqref{eq:s-fin} and $\alpha\in (0,1]$,
then $\sigma$ is the spectral measure of a St$\alpha$S random measure
$\xi$, whose Laplace functional is given by
\eqref{eq:sd}. Therefore
(\ref{spectral_representation_fstable}) can be written as
\begin{displaymath}
  G_\Phi [h]=L_{\xi}\big[1-B(h)\big],
\end{displaymath}
which, by Theorem~\ref{teo15} implies the $\mathcal{F}$-stability of $\Phi$.
\end{proof}

We also get the following generalisation of Theorem~\ref{th:pois-sib}.
\begin{theorem}
  An $\mathcal{F}$-stable point process with a spectral measure
  $\sigma$ supported only by the set $\M_1$ of probability measures
  can be represented as a cluster process with centre process being a
  Poisson process on $\mathbb{M}_1$ driven by the spectral measure
  $\sigma$ and daughter processes having p.g.fl.\
  $G_{\Upsilon(\mu)}[B(h)]$, where $\Upsilon(\mu)$ are
  $\Sib(\alpha,\mu)$ distributed point processes and $B(\cdot)$ is the
  $B$-function of the branching process driven by $\mathcal{F}$. The
  daughter process corresponds to a Sibuya p.p.\ $\Upsilon(\mu)$ with
  its every point given a multiplicity according to independent
  copies of $Y_\infty$.
\end{theorem}
\begin{proof}
  In the case when the spectral measure $\sigma$ is supported by
  probability measures, representation
  (\ref{spectral_representation_fstable}) becomes
  \begin{equation}\label{spectral_representation_fstable_finite}
      G_\Phi [h]=\exp\big\{-\int_{\mathbb{M}_1}\langle 1-B(h),
      \mu\rangle^\alpha \sigma(d\mu)\big\}\qquad\forall h\in\mathcal{V}(\mathcal{X}),
    \end{equation}\
  where $\mathbb{M}_1$ is the space of probability measures on
  $\mathcal{X}$. In terms of the \pgfl \eqref{eq:sibpgfl} of a Sibuya
  p.p., this reads
  \begin{multline}
      G_\Phi [h]=\exp\big\{-\int_{\mathbb{M}_1}1-(1-\langle
      1-B(h),\mu\rangle^\alpha) \sigma(d\mu)\big\}=
      \\
      =\exp\big\{-\int_{\mathbb{M}_1}\big(1-G_{\Upsilon(\mu)}[B(h)]\big)
      \sigma(d\mu)\big\}\qquad
      h\in\mathcal{V}(\mathcal{X}),\label{spectral_representation_fstable_sibuya} 
    \end{multline}\
  where $\Upsilon(\mu)$ denotes a point process following the
  $\Sib(\alpha,\mu)$ distribution. Notice that, since by
  (\ref{AB_relationship}), $B(\cdot)$ is the p.g.f.\ of the
  distribution $Y_\infty$, $G_{\Upsilon(\mu)}[B(h)]$ is the p.g.fl.\
  of a point process by \eqref{eq:clg}.
\end{proof}

As we have seen in \eqref{F_stable_natural}, $\mathcal{F}$-stable
processes are infinitely divisible. The latter can be divided into two
classes: regular and singular depending on whether their KLM-measure
is supported by the set of finite or infinite configurations (\seg
\cite[Def.10.2.VI]{DVJ2}). Similarly to the proof of Theorem~29 in
\cite{DavMolZuy:11} on the decomposition of \das\ processes, we can
extend this result to $\mathcal{F}$-stable processes.
\begin{theorem}\label{das_decomposition}
  An $\mathcal{F}$-stable p.p.\ $\Phi$ with a spectral measure $\sigma$ can
  be represented as the sum of two independent $\mathcal{F}$-stable
  point processes:
  \begin{displaymath}
    \Phi=\Phi_r+\Phi_s,
  \end{displaymath}
  where $\Phi_r$ is regular and $\Phi_s$ singular. $\Phi_r$ is an
  $\mathcal{F}$-stable p.p.\ with spectral measure being $\sigma\big|
  _{\M_1}=\sigma(\cdot\,\cap\M_1)$ and $\Phi_s$ is an
  $\mathcal{F}$-stable p.p.\ with spectral measure $\sigma\big|
  _{\SS\setminus \M_1}$.
\end{theorem}
The regular component $\Phi_r$ can be represented as a cluster
p.p.\ with p.g.fl.\ given by
\begin{displaymath}
  G_{\Phi_r}[h]=\exp\big\{-\int_{\mathbb{M}_1}\big(1-G_{\Upsilon(\mu)}[B(h)]\big)
  \sigma|_{\M_1}(d\mu)\big\}\qquad\forall h\in\mathcal{V}(\mathcal{X}).
\end{displaymath}
On the contrary, the singular component $\Phi_s$ is not a
cluster p.p., and its p.g.fl.\ is given by
\eqref{spectral_representation_fstable} (with $\sigma$ replaced
with $\sigma\big|_{\SS\setminus\M_1}$ there).

\section{General branching stability for point
  processes}\label{sec:general_branching_stability}
Stable distributions appear in various limiting schemes because
decomposition of a sum into a proportion $t$ and $1-t$ of the summands
inevitably leads to the limiting distribution satisfying
\eqref{definition}. We have seen that this argument still works for point
processes when the multiplication is replaced by a stochastic
branching operation, the reason being associativity and distributivity
with respect to the sum (superposition). One may ask: to which extent
one can generalise this stochastic multiplication operation so that it
still satisfies associativity and distributivity? The answer is given
in this section: branching operations are, in this sense, the
exhaustive generalisation.

\subsection{Markov branching processes on $\sN$}

Markov branching processes on $\sN$ (also called branching
diffusions or branching particle systems) basically consist of a
diffusion component and a branching component: each particle,
independently of the others, moves according to a diffusion process
and after an exponential time it branches. When a particle branches it is
replaced by a random configuration of points (possibly empty, in which
case the particle dies) depending on the location of the particle at
the branching time (e.g.\ \cite{AH:83}, \cite{D:94} or \cite{Etheridge:00}).

Alternatively branching particle systems can be defined as Markov
processes satisfying the branching property, as follows.

\begin{definition}\label{defi:gmbp}
  \emph{A Markov branching process} on $\sN$ is a time-homogeneous
  Markov process $(\Psi_t^\varphi)_{t\geq0,\,\varphi\in\sN}$ on
  $(\sN,\mathcal{B}(\sN))$, where $t$ denotes time and $\varphi$ the
  starting configuration, such that its probability transition kernel
  $P_t(\varphi,\cdot)$ satisfies the branching property:
\begin{equation}\label{eq:branching_property_particle_systems}
P_t(\varphi_1+\varphi_2,\cdot)
\;=\;
P_t(\varphi_1,\cdot)*P_t(\varphi_2,\cdot),
\end{equation}\
for any $t\geq 0$ and $\varphi_1$, $\varphi_2$ in $\sN$.
\end{definition}
The branching property \eqref{eq:branching_property_particle_systems}
can also be expressed in terms of p.g.fl.'s as follows:
\begin{equation}\label{condition2_eq}
    G_t^\varphi[h]\;=\left\{
  \begin{array}{ll}
    1, & \hbox{if }\varphi=\zero,\\
    \prod_{x\in\varphi} G_t^{\delta_x}[h], & \hbox{if }\varphi\neq\zero,
  \end{array}
\right.\qquad h\in\mathcal{V}(\mathcal{X}),
\end{equation}\
where $G_t^{\varphi}$ and $G_t^{\delta_x}$ are the p.g.fl.'s of
$\Psi_t^{\varphi}$ and $\Psi_t^{\delta_x}$ respectively  (\seg
\cite[Ch.~5.1]{AH:83} or \cite[Ch. 3]{Dawson1978}). Under some
additional regularity assumption, every Markov branching process on
$\sN$ (defined as above) can be expressed in terms of particles
undergoing diffusion and branching see, for example, \cite{Ikeda:68}.

In general not every starting configuration $\phi\in\sN$ is
allowed. In fact when $\phi$ consists of an infinite number of
particles, the diffusion component could move an infinite number of
particles in a bounded set. Therefore in general one needs to consider
only starting configuration $\phi$ such that $G_t^\varphi[h]<\infty$
for any $t\geq 0$ (for more details see, e.g.,
\cite[Ch. 5]{Dawson1978} or \cite[Ch. 1.8]{Etheridge:00}).

\subsection{General branching operation for point processes}

Let $\bullet \, :\ (t,\Phi)\rightarrow t\bullet\Phi,\ t\in[0,1]$ be a
stochastic operation acting on point processes on $\mathcal{X}$ or,
more exactly, on their distributions.  We assume $\bullet$ to act
independently on each realisation of the point process, meaning
that
\begin{equation}\label{eq:A1}\tag{A1}
  \P(t\bullet\Phi\in A)=\int_{\sN}\P(t\bullet\varphi\in
  A)\P_\Phi(d\varphi),\qquad A\in\mathcal{B}(\sN),\; t\in(0,1], 
\end{equation}\
where $\P_{\Phi}$ is the distribution of $\Phi$.
We require $\bullet$ to be associative and distributive with respect
to superposition: for any $t,t_1,t_2\in(0,1]$ and $\Phi,\Phi_1,\Phi_2$
independent p.p.'s on $\sN$
\begin{gather}
\label{eq:associativity}\tag{A2}
t_1\bullet(t_2\bullet\Phi)		\quad\stackrel{\mathcal{D}}=\quad
(t_1t_2)\bullet\Phi			\quad\stackrel{\mathcal{D}}=\quad
t_2\bullet(t_1\bullet\Phi),\\
\label{eq:distributivity}\tag{A3}
t\bullet(\Phi_1+\Phi_2)		\quad\stackrel{\mathcal{D}}=\quad
t\bullet\Phi_1+t\bullet\Phi_2,
\end{gather}\
where in \eqref{eq:associativity} and \eqref{eq:distributivity}
the different instances of the $\bullet$ operation are
performed independently.
Note that (A3) implies that $t\bullet \zero\stackrel{\mathcal{D}}=\zero$ for
all $t\in(0,1]$, where $\zero$ is the empty configuration.

\begin{remark}\label{rmk:configurations}
  Because of \eqref{eq:A1}, $\bullet$ is uniquely defined by its
  actions on deterministic configurations $\varphi\in\sN$. In fact,
  given \eqref{eq:A1}, $\Phi,\Phi_1,\Phi_2$ in
  \eqref{eq:associativity} and \eqref{eq:distributivity} can be
  replaced with deterministic point configurations
  $\varphi,\varphi_1,\varphi_2\in\mathcal{N}$. Note that,
  although $\varphi$ is deterministic, $t\bullet\varphi$ is generally 
  stochastic (as, for example, for the thinning operation).
\end{remark}

As we have
seen in \eqref{eq:fass} and \eqref{eq:fcom}, the (local) branching operation
$\circ_{\mathcal{F}}$ operation satisfy
\eqref{eq:A1}-\eqref{eq:distributivity}.
The following results characterises stochastic operations satisfying
\eqref{eq:A1}-\eqref{eq:distributivity} in terms of Markov branching
processes on $\sN$.
\begin{definition}
  We call a stochastic operation $\bullet$  acting on p.p.'s on $\mathcal{X}$ a
  \emph{(general) branching operation} if there exist a Markov branching
  process $(\Psi_t^\varphi)_{t\geq0,\,\varphi\in\sN}$ on $(\sN,\mathcal{B}(\sN))$,
  such that for any p.p. $\Phi$ on $\mathcal{X}$
\begin{equation}\label{passage_operation_branching}
\Psi_t^\Phi\stackrel{\mathcal{D}}=e^{-t}\bullet\Phi\qquad t\in[0,+\infty).
\end{equation}\
\end{definition}
\begin{proposition}\label{maximality_of_the_operation}
  A stochastic operation $\bullet$ satisfies
  \eqref{eq:A1}-\eqref{eq:distributivity} if and only if it is a
  general branching operation.
\end{proposition}
\begin{proof}
  \textit{Necessity:} Let $\bullet$ satisfy
  \eqref{eq:A1}-\eqref{eq:distributivity}. Let $P_t(\varphi,\cdot)$ denote
  the distribution of $e^{-t}\bullet\varphi$. By putting
  $\Phi=e^{-t_1}\bullet\psi$, $\psi\in\sN$, in \eqref{eq:A1} we obtain
\begin{equation}\label{eq:maximality_proof1}
  \P\big[e^{-t_2}\bullet(e^{-t_1}\bullet\psi)\in A\big]
  =\int_{\sN}P_{t_2}(\varphi,A)P_{t_1}(\psi,d\varphi),
  \qquad A\in\mathcal{B}(\sN),\; t_1,t_2>0.
\end{equation}\
Using the associativity of $\bullet$
(Assumption~\eqref{eq:associativity}) on the left-hand side of
\eqref{eq:maximality_proof1} we obtain the Chapman-Kolmogorov
equations
\begin{equation}\label{eq:CK}
  P_{t_1+t_2}(\varphi,A)\;=\;\int_{\sN}P_{t_1}(\psi,A)P_{t_2}(\varphi,d\psi)
  \qquad A\in\mathcal{B}(\sN),\; t_1,t_2>0.
\end{equation}\
Therefore, by the Kolmogorov extension theorem there exists
 a Markov process $\Psi_t^\varphi$ on $\sN$
having transition kernel $\big(P_t(\varphi,\cdot)\big)_{t\geq 0}$. Let
$\varphi\in\sN\backslash\zero$ and $G_t^\varphi[\cdot]$ be the
p.g.fl.\ of $P_t(\varphi,\cdot)$ for $t>0$. Using the definition of
$P_t(\varphi,\cdot)$ and the distributivity of $\bullet$ (Assumption
\eqref{eq:distributivity}) we obtain
\begin{displaymath}
  G_t^\varphi[h]=
  G_{e^{-t}\bullet\varphi}[h]=
  G_{\sum_{x\in\varphi}e^{-t}\bullet\delta_x}[h]=
  \prod_{x\in\varphi} G_{e^{-t}\bullet\delta_x}[h]=
  \prod_{x\in\varphi} G_t^{\delta_x}[h],\qquad h\in\mathcal{V}(\mathcal{X}).
\end{displaymath}
From distributivity it also follows that
$e^{-t}\bullet\zero\stackrel{\mathcal{D}}=\zero$ and therefore
$G_t^\zero[h]=1$ for any $h\in\mathcal{V}(\mathcal{X})$. Therefore
\eqref{condition2_eq} is satisfied and $\Psi_t^\varphi$ is a Markov
branching process on $\sN$.

\textit{Sufficiency:} Let $(\Psi_t^\varphi)_{t\geq0},\ \varphi\in\sN$ be a
Markov branching process on $\sN$ with transition kernel
$P_t(\varphi,\cdot)$. Consider the operation $\bullet$ induced by
\eqref{passage_operation_branching}, namely
$t\bullet\Phi\stackrel{D}=\Psi_{-\ln t}^\Phi$. Assumption
\eqref{eq:A1} follows by the construction:
\begin{displaymath}
  \P(\Psi_t^\Phi\in A)=\int_{\sN}\P(\Psi_t^\varphi\in
  A)\P_\Phi(d\varphi),\qquad A\in\mathcal{B}(\sN),\ t>0. 
\end{displaymath}
Given $\varphi\in\sN$ and $t_1,t_2\in (0,1]$, using \eqref{eq:A1} and
the Chapman-Kolmogorov equations \eqref{eq:CK} we obtain
\begin{multline*}
\P\left(t_1\bullet(t_2\bullet \varphi)\in A\right)
\;=\;
\int_{\sN} P_{-\ln t_1}(\psi,A)P_{-\ln t_2}(\varphi,d\psi)
\;=\\=\;
P_{-\ln (t_1t_2)}(\varphi,A)
\;=\;
\P((t_1t_2)\bullet \varphi\in A) \qquad A\in\mathcal{B}(\sN),
\end{multline*}\
i.e.\ the associativity \eqref{eq:associativity}) of $\bullet$ holds.

Finally, let $G_t^\varphi[\cdot]$ be the p.g.fl. of $\Psi_t^\varphi$ for
$t\geq0$ and $\varphi\in\mathcal{N}$. Given $t\in(0,1]$ and
$\varphi_1,\varphi_2\in\sN\backslash\zero$, using the independent
branching property \eqref{condition2_eq}, it
follows that
\begin{multline}\label{eq:distrib_pgfl_proof}
  G_{t\bullet(\varphi_1+\varphi_2)}[h] = G_{-\ln
    t}^{\varphi_1+\varphi_2}[h]= \prod_{x\in\varphi_1+\varphi_2}
  G_t^{\delta_x}[h]= \prod_{x\in\varphi_1}
  G_{t}^{\delta_x}[h]\;\prod_{x\in\varphi_2}
  G_{t}^{\delta_x}[h]=\\
   = G_{-\ln t}^{\varphi_1}[h]\;G_{-\ln t}^{\varphi_2}[h]= G_{t\bullet
    \varphi_1}[h]G_{t\bullet\varphi_2}[h]\qquad
  h\in\mathcal{V}(\mathcal{X}).%
\end{multline}\
The distributivity \eqref{eq:distributivity} of $\bullet$ follows from
\eqref{eq:distrib_pgfl_proof} and Remark~\ref{rmk:configurations}.
\end{proof}

\begin{example}[Diffusion] \label{ex:diffusion} Let $(X_t)_{t\geq 0}$
  be a strong time-homogeneous Markov process on $\mathcal{X}$, right
  continuous with left limits. Let $(\Psi_t^\varphi)_{t\geq0,\,\varphi\in\sN}$ be the
  Markov branching process on $\sN$ where every particle moves
  according to an independent copy of $X_t$, without branching (see
  \cite[Sec.~V.1]{AH:83} for a proof that this is indeed a Markov
  branching process on $\sN$). Denote by $\bullet_d$ the associated
  branching operation, $t\bullet_d\Phi\stackrel{D}=\Psi_{-\ln
    t}^\Phi$.  Let $P_t(x,\cdot)$ be the distribution of $X_t^x$,
  where $x$ denotes the starting state, and $P_th(x)=\E h(X_t^x)
  =\int_{\mathcal{X}}h(y)P_t(x,dy)$.  Then, for any $\phi$ in $\sN$ we
  have
\begin{multline}\label{eq:diffusion}
  G_{t\bullet_d\phi}[h]=
  G_{-\ln t}^\phi[h]=
  \prod_{x\in\phi} G_{-\ln t}^{\delta_x}[h]=\\=
  \prod_{x\in\phi}\E h(X_{-\ln t}^x)=
  \prod_{x\in\phi}P_{-\ln t}h(x)=
  G_{\phi}[P_{-\ln t}h]
  \qquad h\in\mathcal{V}(\mathcal{X}).
\end{multline}\
\end{example}

\begin{example}[Diffusion with thinning] \label{ex:thinning_diffusion}
  Let $X_t$ be as in Example~\ref{ex:thinning_diffusion}. Let
  $(\Psi_{t}^{\varphi})_{t\geq0,\,\varphi\in\sN}$ be the Markov Branching
  process on $\sN$, where every particle moves according to an
  independent copy of $X_t$ and after an exponentially
  $\mathrm{Exp}(1)$-distributed time it dies (independently of the other
  particles). We denote by $\bullet_{dt}$ the associated branching
  operation $t\bullet_{dt}\Phi\stackrel{D}=\Psi_{-\ln
    t}^\Phi$. Similarly to \eqref{eq:diffusion}, it is easy
  to show that given a p.p.\ $\Phi$ we
  have
\begin{equation*}
  G_{t\bullet_{dt}\Phi}[h]=G_\Phi[1-t+t(P_{-\ln t}h)],\qquad h\in \mathcal{V}(\mathcal{X}).
\end{equation*}\
This operation acts as the composition of the thinning operation
$\circ$ and the diffusion operation $\bullet_d$ introduced in
Example~\ref{ex:diffusion}, regardless of the order in which these two
operations are applied. For any p.p.\ $\Phi$
\begin{displaymath}
  t\bullet_{dt}\Phi\stackrel{\mathcal{D}}=t\bullet_d(t\circ\Phi)\stackrel{\mathcal{D}}=t\circ(t\bullet_d\Phi).
\end{displaymath}
In fact, since
$1-t+t(P_{-\ln t}h)=P_{-\ln t}(1-t+th)$ for any $h$ in
$\mathcal{V}(\mathcal{X})$, it holds
\begin{multline*}
G_{t\bullet_d(t\circ\Phi)}[h]=
G_{t\circ\Phi}[P_{-\ln t}h]=
G_\Phi[1-t+t(P_{-\ln t}h)]=
G_{t\bullet_{dt}\Phi}[h]=
  \\%
=G_\Phi[P_{-\ln t}(1-t+th)]=
G_{t\bullet_d\Phi}[1-t+th]=
G_{t\circ(t\bullet_d\Phi)}[h].
\end{multline*}\
\end{example}

\subsection{Stability for general branching operations}
Proposition~\ref{maximality_of_the_operation} shows that branching
operations are the only operations on point processes satisfying
assumptions \eqref{eq:A1}-\eqref{eq:distributivity}. Such assumptions,
together with the continuity and subcriticality conditions below, lead to
an appropriate definition of stability, as
Proposition~\ref{definitions_of_stability} below shows.  

\begin{definition}
  We call a branching operation $\bullet$ \emph{continuous} if
  \begin{equation}\tag{A4}\label{eq:continuity}
    t\bullet\varphi\Rightarrow\varphi\quad\hbox{ for }t\uparrow
    1\qquad \text{for every }  \varphi\in\mathcal{N},
  \end{equation}\ 
  where $\Rightarrow$ stands for weak convergence (or equivalently for
  the convergence in Prokhorov metric). Moreover we say that $\bullet$
  is \emph{subcritical} if the associated Markov branching process on
  $\sN$, $\Psi_t^\varphi$, is subcritical, i.e.\ if
  $\E\Psi_t^{\delta_x}(\mathcal{X})<1$ for every $x\in\sX$ and $t>0$.
\end{definition}

\begin{proposition}\label{definitions_of_stability}
  Let $\Phi$ be a p.p.\ on $\mathcal{X}$ with p.g.fl.\ $G_\Phi[\cdot]$
  and $\bullet$ be a subcritical and continuous branching operation
  on $\mathcal{X}$. Then the following conditions are
  equivalent:
  \begin{enumerate}\itemsep0pt
  \item $\forall$ $n\in\mathbb{N}$ $\exists$ $c_n\in(0,1]$ such that
    given $(\Phi^{(1)},...,\Phi^{(n)})$ independent copies of $\Phi$
    \begin{equation}\label{eq:first_defi}
      \Phi \stackrel{\mathcal{D}}= c_n\bullet (\Phi^{(1)}+...+\Phi^{(n)});
    \end{equation}\
  \item $\forall$ $\lambda>0$ $\exists$ $t\in(0,1]$ such that
    \begin{displaymath}
      G_{\Phi}[h]=\big(G_{t\bullet \Phi}[h]\big)^\lambda;
    \end{displaymath}
  \item $\exists$ $\alpha>0$ such that $\forall$ $n\in\mathbb{N}$, given
  $(\Phi^{(1)},...,\Phi^{(n)})$ independent copies of $\Phi$
  \begin{equation}\label{defi_stability3}
    \Phi \stackrel{\mathcal{D}}= (n^{-\frac{1}{\alpha}})\bullet (\Phi^{(1)}+...+\Phi^{(n)});
  \end{equation}\
\item $\exists$ $\alpha>0$ such that $\forall$ $t\in[0,1]$
  \begin{equation}\label{defi_stability4}
    G_{\Phi}[h]=\big(G_{t\bullet \Phi}[h]\big)^{t^{-\alpha}};
  \end{equation}\

\item $\exists$ $\alpha>0$ such that $\forall$ $t\in[0,1]$, given
  $\Phi^{(1)}$ and $\Phi^{(2)}$ independent copies of $\Phi$,
  \begin{equation}\label{defi_stability5}
    t^{1/\alpha}\bullet\Phi^{(1)}+(1-t)^{1/\alpha}\bullet\Phi^{(2)}
    \stackrel{\mathcal{D}}=\Phi. 
  \end{equation}\
\end{enumerate}
\end{proposition}
\begin{proof}
  If $\Phi\equiv\zero$ then all the conditions are trivially
  satisfied. So we suppose $\Phi\not\equiv\zero$.
  $\textit{4)}\Rightarrow \textit{2)}\Rightarrow \textit{1)}$ are
  obvious implications. So if one proves
  \textit{1)}$\Rightarrow$\textit{4)} then \textit{1)}, \textit{2)}
  and \textit{4)} are equivalent.

  To show \textit{1)}$\Rightarrow$\textit{4)} note that, given
  $n\in\mathbb{N}$, the coefficient $c_n$ satisfying
  \eqref{eq:first_defi} is unique. In fact if $c_n$ and $t\,c_n$ both
  satisfy \eqref{eq:first_defi}, with $t\in(0,1)$, by associativity it
  follows 
\begin{equation*}
\Phi
\deq
(tc_n)\bullet (\Phi^{(1)}+...+\Phi^{(n)})
\deq
t\bullet\left(c_n\bullet (\Phi^{(1)}+...+\Phi^{(n)})
\right)
\deq
t\bullet\Phi
\end{equation*}\
and thus, because of subcriticality, $t=1$.
  Using
  \eqref{eq:first_defi} and the distributivity and associativity of
  $\bullet$ we obtain that, given $m,n\in\mathbb{N}$,
\begin{multline*}
  \Phi \;\stackrel{\mathcal{D}}=
  c_n\bullet (\Phi^{(1)}+...+\Phi^{(n)})\deq\\
  \deq c_n\bullet \big(c_m\bullet 
  (\Phi^{(1)}+...+\Phi^{(m)})+...+c_m\bullet
  (\Phi^{(n-1)m+1}+...+\Phi^{(nm)})\big)\\
  \deq(c_nc_m)\bullet (\Phi^{(1)}+...+\Phi^{(nm)}),
\end{multline*}\
which implies that
\begin{equation}\label{cncm}
c_{nm}=c_nc_m.
\end{equation}\
Since we are considering the subcritical case, we have
\begin{equation}\label{cn_smaller_then_cm}
  n>m \Rightarrow c_n<c_m.
\end{equation}\
For every $1\leq m\leq n<+\infty$, $m,n\in\mathbb{N}$ define a
function $c:[1,+\infty)\cap\mathbb{Q}\rightarrow (0,1]$ by setting
\begin{equation}\label{defi_c_function}
  c\Big(\frac{n}{m}\Big):=\frac{c_n}{c_m}.
\end{equation}\
The function $c$ is well defined because of \eqref{cncm} 
and it takes values in $(0,1]$ because of \eqref{cn_smaller_then_cm}.
Using associativity, distributivity and \eqref{eq:first_defi},
\begin{multline}
  \big(G_{\frac{c_n}{c_m}\bullet\Phi}[h]\big)^{\frac{n}{m}}=
  \big(G_{\frac{c_n}{c_m}\bullet\big(c_m\bullet(\Phi^{(1)}+...+\Phi^{(m)})\big)}[h]\big)^{\frac{n}{m}}\;=\\=\;
  \big(G_{c_n\bullet(\Phi^{(1)}+...+\Phi^{(m)})}[h]\big)^{\frac{n}{m}}=
  \Big(\big(G_{c_n\bullet\Phi}[h]\big)^m\Big)^{\frac{n}{m}}=
  \big(G_{c_n\bullet\Phi}[h]\big)^{n}= G_\Phi[h].
\end{multline}\
Therefore
\begin{equation}\label{rational_case}
  G_\Phi[h]=\big(G_{c(x)\bullet\Phi}[h]\big)^{x}\quad\forall x\in[1,+\infty)\cap\mathbb{Q}.
\end{equation}\
It follows from \eqref{cn_smaller_then_cm} and \eqref{defi_c_function}
that $c$ is a strictly decreasing function on
$[1,+\infty)\cap\mathbb{Q}$. Therefore we can be extended to the whole
$[1,+\infty)$ by putting
\begin{displaymath}
  c(x):=\inf\{c(y):\ y\in[1,x]\cap\mathbb{Q}\}.
\end{displaymath}

From \eqref{cncm} and \eqref{cn_smaller_then_cm}, taking limits over
rational numbers, if follows that $c(xy)=c(x)c(y)$ for every
$x,y\in[1,+\infty)$. The only monotone functions $c$ from
$[1,+\infty)$ to $(0,1]$ such that $c(0)=1$ and $c(xy)=c(x)c(y)$ for
every $x,y\in[1,+\infty)$ are $c(x)=x^r$ for some $r\in\mathbb{R}$. Since
our function is decreasing then $r<0$. Let $\alpha>0$ be such that
$r=-1/\alpha$. Fix $x\in[1,+\infty)$ and let
$\{x_n\}_{n\in\mathbb{N}}\subset[1,+\infty)\cap\mathbb{Q}$ be such
that $x_n\downarrow x$ as $n\rightarrow +\infty$, and therefore
$x_n^{-1/\alpha}\uparrow x^{-1/\alpha}$ as
$n\rightarrow +\infty$. Since $\bullet$ is left-continuous in the weak
topology (assumption \eqref{eq:continuity}) it holds that
\begin{displaymath}
  x_n^{-1/\alpha}\bullet\Phi\Rightarrow
  x^{-1/\alpha}\bullet
\Phi\qquad n\rightarrow +\infty,
\end{displaymath}
which implies 
\begin{displaymath}
  G_{x_n^{-1/\alpha}\bullet\Phi}[h]
\quad\stackrel{n\rightarrow\infty}\longrightarrow\quad
G_{x^{-1/\alpha}\bullet
\Phi}[h]
\qquad \forall\; h\in \mathcal{V}(\mathcal{X}).
\end{displaymath}
From \eqref{rational_case} we have
\begin{displaymath}
  \big(G_\Phi[h]\big)^{1/x}=\lim_{n\rightarrow
    +\infty}G_{c(x_n)\bullet\Phi}[h]
  =\lim_{n\rightarrow +\infty}G_{x_n^{-1/\alpha}\bullet\Phi}[h]
\qquad \forall\; h\in \mathcal{V}(\mathcal{X}),
\end{displaymath}
and therefore we obtain \eqref{defi_stability4} as desired.

$\textit{4)}\Rightarrow \textit{3)}\Rightarrow\textit{1)}$ are obvious
implications and thus \textit{3)} is equivalent to \textit{1)},
\textit{2)} and \textit{4)}.

To show \textit{4) $\Rightarrow$ 5)} take $x,y\in[1,+\infty)$. Then, because
of \textit{4)},
\begin{multline}\label{eq:splitting}
  G_\Phi[h]=G_{(x+y)^{-1/\alpha}\bullet\Phi}[h]^{x+y}
  =G_{x^{-1/\alpha}\big(\frac{x+y}{x}\big)^{-1/\alpha}\bullet\Phi}[h]^x
  \cdot G_{y^{-1/\alpha}\big(\frac{x+y}{y}\big)^{-1/\alpha}\bullet\Phi}[h]^y=\\
  =G_{\big(\frac{x+y}{x}\big)^{-1/\alpha}\bullet\Phi}[h]
  \cdot G_{\big(\frac{x+y}{y}\big)^{-1/\alpha}\bullet\Phi}[h]=
  G_{\big(\frac{x+y}{x}\big)^{-1/\alpha}\bullet\Phi
    +\big(\frac{x+y}{y}\big)^{-1/\alpha}\bullet\Phi'}[h],
\end{multline}\
where $\Phi'$ is an independent copy of $\Phi$. Then \textit{5)}
follows since $x,y\in[1,+\infty)$ are arbitrary.

\textit{5)$\Rightarrow$3)}. \eqref{defi_stability3} can be obtained 
iterating \eqref{defi_stability5} n-1 times.
\end{proof}

Section~\ref{sec:general_branching_stability} shows that branching
operations are the most general class of associative and distributive
operations that can be used to study stability for point
processes. Therefore the following definition generalises all the
notions of discrete stability considered so far.
\begin{definition}
  Let $\Phi$ and $\bullet$ be as in
  Proposition~\ref{definitions_of_stability}. If
  \eqref{defi_stability5} is satisfied we say that $\Phi$ is strictly
  $\alpha$-stable with respect to $\bullet$ or, simply, \emph{branching-stable}.
\end{definition}

In this paper we do not provide a characterisation of stable point
processes with respect to a general branching operation. Instead, next we
consider some specific cases that point towards directions to obtain
such a characterisation in full generality.  The main idea is that, given a
branching operation $\bullet$ acting on point processes, there is a
corresponding branching operation $\odot$ acting on random measures
such that stable point processes with respect to $\bullet$ are Cox
processes driven by stable random measures with respect to $\odot$.

\subsection{Stability with respect to thinning and diffusion}\label{sec:thinning_diffusion}
\paragraph{Cox characterisation}
Recall that \das\ point processes are Cox processes driven by \sas\
intensity measures, see Section~\ref{sec:das_pp}. The main reason for
this is that the thinned version of a Poisson p.p.\ with intensity
measure $\mu$, $\Pi_\mu$, is itself a Poisson p.p.\ with intensity
measure $t\mu$, i.e.\
$t\circ\Pi_\mu\stackrel{\mathcal{D}}=\Pi_{t\mu}$.
The same holds for a Cox p.p.\ $\Pi_\xi$ driven by a random measure $\xi$:
\begin{equation}\label{eq:thinning_multiplication}
  t\circ\Pi_\xi\stackrel{\mathcal{D}}=\Pi_{t\xi}. 
\end{equation}\

For the thinning and diffusion operation $\bullet_{dt}$ of Example \ref{ex:thinning_diffusion},
\begin{multline}
  G_{t\bullet_{dt}\Pi_\mu}[h]=G_{\Pi_\mu}[1-t+tP_{-\ln t}h]
  =\exp\{-\langle1-(1-t+tP_{-\ln t}h),\mu\rangle\}=\\
  =\exp\{-\langle tP_{-\ln t}(1-h),\mu\rangle\}=\exp\{-\langle
  1-h,tP^{*}_{-\ln t}\mu\rangle\}=G_{\Pi_{tP^{*}_{-\ln
        t}\mu}}[h],\label{eq:td_pgfl}
\end{multline}\
where $P^*_{-\ln t}$ is the adjoint to the linear operator $P_{-\ln t}$. If we
denote by $\odot_{dt}$ the following
operation
\begin{equation}
\begin{split}
\odot_{dt}:(0,1]\times\mathcal{M}& \rightarrow \mathcal{M}\\
(t,\mu)\quad & \rightarrow t\odot_{dt}\mu\bydef tP^{*}_{-\ln t}\mu,
\end{split}
\end{equation}\
then \eqref{eq:td_pgfl} implies
$t\bullet_{dt}\Pi_\mu\deq \Pi_{t\odot_{dt}\mu}$. Similarly
for Cox processes,
\begin{equation}\label{eq:td_convmult}
t\bullet_{dt}\Pi_\xi\deq \Pi_{t\odot_{dt}\xi},
\end{equation}\
where the operation $\odot_{dt}$ acts on each realisation of $\xi$.
The analogy between \eqref{eq:thinning_multiplication} and \eqref{eq:td_convmult}
suggests the following result.

\begin{proposition}\label{prop:cox_characterization_for diffusion}
A point process $\Phi$ is strictly $\alpha$-stable with respect to
$\bullet_{dt}$ if and only if it is a Cox process $\Pi_\xi$ with an
intensity measure being strictly $\alpha$-stable with respect to
$\odot_{dt}$. 
\end{proposition}
\begin{proof}
  \emph{Sufficiency}. Suppose $\xi$ is strictly $\alpha$-stable with
  respect to $\odot_{dt}$. Then using \eqref{eq:td_convmult} and the
  stability of $\xi$
\begin{multline*}
  t^{1/\alpha}\bullet_{dt}\Pi'_\xi+(1-t)^{1/\alpha}\bullet_{dt}\Pi''_\xi
  \stackrel{\mathcal{D}}=
  \Pi'_{t^{1/\alpha}\odot_{dt}\xi}+\Pi''_{(1-t)^{1/\alpha}\odot_{dt}\xi}
  \stackrel{\mathcal{D}}=
  \Pi_{t^{1/\alpha}\odot_{dt}\xi'+(1-t)^{1/\alpha}\odot_{dt}\xi''}
  \stackrel{\mathcal{D}}= \Pi_{\xi},
\end{multline*}\
where the $\Pi_\xi'$ and $\Pi_\xi''$ are independent copies of
$\Pi_\xi$. Therefore $\Pi_\xi$ is strictly $\alpha$-stable with respect to $\bullet_{dt}$.\\
\emph{Necessity}. Suppose $\Phi$ is strictly $\alpha$-stable with
respect to $\bullet_{dt}$. From \eqref{defi_stability3} we have that
for any positive integer $m$
\begin{equation}\label{eq:g_phi_iter}
  G_\Phi[h]=\left(G_{m^{-1/\alpha}\bullet_{dt}\Phi}[h]\right)^m
  =\left(G_{\Phi}[1-m^{-1/\alpha}+m^{-1/\alpha}
    P_{\frac{\log m}{\alpha}}h]\right)^m,\qquad h\in\sV(\sX).
\end{equation}\
We need to show that
$G_{\Phi}[1-u]$, as a functional of $u\in BM_+(\sX)$,
is the Laplace functional of a random measure $\xi$ that is strictly
$\alpha$-stable with respect to $\odot_{dt}$. Since $1-u$ may not take
values in $[0,1]$, the expression $G_{\Phi}[1-u]$ may not
be well defined. Thus we use \eqref{eq:g_phi_iter} and define $L[u]$
as
\begin{equation}\label{eq:l_defi_2}
L[u]=\left(G_{\Phi}[1-m^{-1/\alpha}P_{\frac{\log m}{\alpha}}u]\right)^m,\qquad u\in BM_+(\sX),
\end{equation}\
noting that for some $m$ big enough
$1-m^{-1/\alpha}P_{\frac{\log m}{\alpha}}u$ takes values in $[0,1]$
and the right-hand side of \eqref{eq:l_defi_2} is well
defined. Arguing as in the proof of Theorem 3.5 one can prove that $L$ is the Laplace functional of a random measure
$\xi$, $L_\xi$. Finally, $\xi$ is strictly $\alpha$-stable with
respect to $\odot_{dt}$ because for any $u$ in $BM_+(\sX)$
\begin{multline*}
L_{t\odot_{dt}\xi}[u]=
\E\exp\{-\langle u,t\odot_{dt}\xi\rangle\}=
\E\exp\{-\langle u,tP^*_{-\ln t}\xi\rangle\}=\\
=\E\exp\{-\langle tP_{-\ln t}u,\xi\rangle\}=
L_{\xi}[tP_{-\ln t}u],
\end{multline*}\
and, supposing $u\leq t^{-1}$, and thus $(1-tP_{-\ln t}u)\in\mathcal{V}(\mathcal{X})$, we have
\begin{multline}\label{eq:sas_dt}
L_{\xi}[tP_{-\ln t}u]
\,=\,
G_{\Phi}[1-tP_{-\ln t}u]
\,=\,
G_{\Phi}[1-t+tP_{-\ln t}(1-u)]
\,=\\
G_{t\bullet_{dt}\Phi}[1-u]
\,=\,
G_{\Phi}[1-u]^{t^{\alpha}}
\,=\,
L_{\xi}[u]^{t^{\alpha}}.
\end{multline}\
Similar calculations also apply to the general definition of $L[u]$ in \eqref{eq:l_defi_2}, which includes the case $u> t^{-1}$.
The fact that $\xi$ is strictly $\alpha$-stable with respect to $\odot_{dt}$ follows from \eqref{eq:sas_dt} arguing, for example, as in \eqref{eq:splitting}.
\end{proof}

\paragraph{Levy characterisation and spectral decomposition}\label{sec:spectral_decomposition}
Proposition \ref{prop:cox_characterization_for diffusion}
characterises stable p.p.'s with respect to $\bullet_{dt}$ as Cox
processes driven by stable random measures with respect to
$\odot_{dt}$.  In this section we describe stable random measures with
respect to $\odot_{dt}$ in terms of homogeneous Levy measures (with
respect to $\odot_{dt}$) and we show how to decompose such homogeneous
Levy measures into a spectral and a radial component.

Given $A\in\mathcal{B}\big(\mathcal{M}\big)$ and $t\in(0,1]$ we define
$t\odot_{dt} A=\{t\odot_{dt}\mu:\hspace{1.5mm}\mu\in A\}$.
The idea is to look for homogeneous Levy measures of order $\alpha$
with respect to $\odot_{dt}$, meaning that for any $A$ in
$\mathcal{B}\big(\mathcal{M}\big)$
\begin{equation}\label{eq:new_homogeneity}
\Lambda(t\odot_{dt} A)=t^{-\alpha}\Lambda(A)\quad\forall t\in(0,1].
\end{equation}\

Let us consider Laplace functionals of the form
\begin{equation}\label{eq:levy_characterization}
L[h]=\exp\left\{
-\int_{\sM\backslash\{0\}} \left(1-e^{-\langle h,\mu \rangle}\right)\Lambda(d\mu)
\right\},\qquad h\in BM_+(\sX),
\end{equation}\
where $\Lambda$ is a Radon measure on $\sM\backslash\{0\}$ such that
\begin{equation}
\int_{\sM\backslash\{0\}}\left(1-e^{-\langle h,\mu \rangle}\right)\Lambda(d\mu)<\infty,
\end{equation}\
for any $h$ in $BM_+(\sX)$, and \eqref{eq:new_homogeneity} holds for
any $A$ in $\mathcal{B}\big(\mathcal{M}\big)$. Arguing as in the proof
of Theorem~2 of \cite{DavMolZuy:11}, it can be seen that
\eqref{eq:levy_characterization} defines the Laplace functional of a
random measure, say $\xi$. Then, defining $L_\xi[h]$ as in
\eqref{eq:levy_characterization}, from \eqref{eq:new_homogeneity} it
follows that
\begin{equation*}
L_{t\odot_{dt}\xi}[h]=\exp\left\{
-\int_{\sM\backslash\{0\}} \left(1-e^{-\langle h,\mu \rangle}\right)t^{-\alpha}\Lambda(d\mu)
\right\}=L_{\xi}[h]^{t^{-\alpha}},\qquad h\in BM_+(\sX),
\end{equation*}\
which means that $\xi$ is $\alpha$-stable with respect to $\odot_{dt}$ by an argument analogous to the one in \eqref{eq:splitting}.

We now show how to decompose Levy measures satisfying
\eqref{eq:new_homogeneity} in a radial component (uniquely determined
by $\alpha$) and a spectral component. Such a spectral decomposition
depends on the operation $\odot_{dt}$ and thus it is not the one used
in Section \ref{sec:das_pp} for thinning-stable point processes. For
simplicity we restrict ourselves to the case where $\sM$ is the space
of finite measures on $\sX=\R^n$ for some $n$ and the the diffusion
process $P_t$ is a Brownian motion, meaning that given $\mu$ in $\sM$
and $t$ in $(0,1]$, the measure $t\odot_{dt}\mu$
is 
\begin{displaymath}
    t\odot_{dt}\mu\;=\;
    t\,\nu_t*\mu,
\end{displaymath}
where $*$ denotes the convolution of measures and $\nu_t$, for $t$ in
$(0,1)$, has the following density with respect to the Lebesgue
measure
\begin{displaymath}
  \frac{d\nu_t}{d\ell}=f_t(x)
  =\frac{1}{(2\pi\ln t)^{\frac{n}{2}}}\exp\left\{\frac{|x|^2}{-2\ln t}\right\},
\end{displaymath}
while for $t=0$, $\nu_t$ equals $\delta_{\textbf{0}}$, with
$\textbf{0}$ being the origin of $\R^n$.

We now show that $\sM\backslash\{0\}$ can be
decomposed as $\widetilde{\SS}\times(0,1]$, for the following $\widetilde{\SS}$
\begin{equation*}
  \widetilde{\SS}:=\{\mu\in\sM\backslash\{0\}:
  \hspace{1.5mm}\nexists (t,\rho )\in(0,1)\times\sM\backslash\{0\}
  \hbox{ such that } t\odot_{dt}\rho =\mu\}.
\end{equation*}\

Note that $\widetilde{\SS}\in\sB(\sM)$ because
$\widetilde{\SS}=\cup_{t\in(0,1)\cap
  \mathbb{Q}}t\odot_{dt}(\sM\backslash\{0\})$
and $t\odot_{dt}(\sM\backslash\{0\})\in\sB(\sM)$ for any $t$ in
$(0,1)\cap \mathbb{Q}$. In fact if
$A\in\mathcal{B}\big(\mathcal{M}\big)$, then also
$t\odot_{dt} A\in\mathcal{B}\big(\mathcal{M}\big)$ because the map
$\mu\rightarrow t\odot_{dt}\mu$ is injective (see Proposition
\ref{prop:spectral_decomp} below) and therefore the image of a Borel set is
still Borel (\seg Section 15.A of \cite{Kechris1995}).

\begin{proposition}\label{prop:spectral_decomp}
The map
\begin{equation}
\begin{split}
(0,1]\times\widetilde{\SS}& \mapsto \sM\backslash\{0\}\\
(t,\mu_0)\quad & \mapsto t\odot_{dt} \mu_0,
\end{split}
\end{equation}\
is a bijection.
\end{proposition}
\begin{proof}
  \textit{Injectivity:} First note that given $t\in (0,1)$ and
  $\mu,\rho\in\sM$ such that $\nu_t*\mu=\nu_t*\rho$, then $\mu=\rho$. This
  follows, for example, by noting that, for any $t$ in $(0,1]$, the
  moment generating function of $\nu_t$ is strictly positive on $\R^n$.
Then, given $t_1,t_2\in(0,1]$ and $\mu_0,\rho_0\in\widetilde{\SS}$
\begin{multline*}
  t_1\odot_{dt} \mu_0 =(t_1t_2)\odot_{dt} \rho_0 \Leftrightarrow
  t_1(\nu_{t_1}*\mu_0 ) =(t_1t_2)(\nu_{t_1t_2}*\rho_0 ) \Leftrightarrow\\
  \Leftrightarrow \nu_{t_1}*\mu_0 =\nu_{t_1}*\big(\nu_{t_2}*(t_2\rho_0
  )\big)
\Leftrightarrow\mu_0 =\nu_{t_2}*(t_2\rho_0 )\Leftrightarrow t_2=1\hbox{ and }\mu_0 =\rho_0 ,
\end{multline*}\
which means that the map $(t,\mu_0)\mapsto t\odot_{dt} \mu_0$ is injective.\\
\textit{Surjectivity:} Let $\mu\in\mathcal{M}\setminus\{0\}$.
Without loss of generality we can consider $\mu$
to be a probability measure, otherwise consider
$\frac{\mu}{\mu(\sX)}$. Define
\begin{displaymath}
  I_\mu:=\{t\in(0,1]:\hspace{1.5mm}\exists\rho\in\mathcal{M}
  \setminus\{0\},t\in(0,1]\hbox{ such that }t\odot_{dt} \rho=\mu\}.
\end{displaymath}
Note that $1\in I_\mu$ because $1\odot_{dt}\mu=\mu$
and, thanks to the associativity of $\odot_{dt}$,
$I_\mu$ is an interval. We define $t_0:=\inf
I_\mu$ and we prove $t_0>0$. Suppose $t_0=0$.
Then for any $\epsilon\in (0,1]$ there is $\rho^{(\epsilon )}\in\mathcal{M}
  \setminus\{0\}$ such that
  $\mu=\epsilon\nu_\epsilon*\rho^{(\epsilon )}$.
  Setting $\mu^{(\epsilon )}=\epsilon\rho^{(\epsilon )}$ we have $\mu=\nu_\epsilon*\mu^{(\epsilon )}$.
Since both $\mu$ and $\mu ^{(\epsilon)}$ by
associativity of $\odot_{dt}$ are obtained by convolution with some
$\nu_t$, they both admit bounded density functions, say, $g$ and
$g^{(\epsilon)}$, respectively.
Moreover $\| g^{(\epsilon)}\|_1=1$ because
\begin{displaymath}
1=\| g\|_1=
\| f_\epsilon *g^{(\epsilon )}\|_1
=\int_{\R^n}\int_{\R^n}f_\epsilon(x-y)g^{(\epsilon )}(y)dy\,dx
=\int_{\R^n}g^{(\epsilon )}(y)dy
=\| g^{(\epsilon )}\|_1\,,
\end{displaymath}
were we used the fact that both $f_\epsilon$ and $g^{(\epsilon )}$ are
positive, Fubini's theorem (which holds for positive functions) and
$\| \nu_\epsilon\|_1=1$ and. Then
\begin{displaymath}
\| g\|_\infty=
\| f_\epsilon *g^{(\epsilon )}\|_\infty \leq\| f_\epsilon\|_\infty\| g^{(\epsilon)}\|_1=
\| f_\epsilon\|_\infty\rightarrow 0\quad\hbox{as }\epsilon\rightarrow 0.
\end{displaymath}
Therefore, by contradiction, $t_0>0$.  Given $t\in(t_0,1]$ let
$\mu^{(t)}\in\mathcal{M}\setminus\{0\}$ such that
$\nu_t*\mu^{(t)}=\mu$.  We now prove that $\{\mu^{(t)}\}$ as
$t\downarrow t_0$ is Cauchy in the Prokhorov metric. Let $\epsilon>0$
be fixed and $\delta=\delta(\epsilon)>0$ to be fixed later. Consider
$t_1,t_2$ such that $t_0<t_1<t_2<t_0+\delta\leq 1$.  We have
$\mu^{(t_2)}=\nu_{t_1/t_2}*\mu^{(t_1)}$.  As $\delta\rightarrow 0$ we
have
$\nu_{t_1/ t_2}(\mathbb{R}^n\setminus B_\epsilon (0))\rightarrow
0$.
Thus we can choose $\delta$ such that
$\nu_{t_1/t_2}\left(\mathbb{R}^n\setminus B_\epsilon
  (0)\right)<\epsilon$.  Therefore we have
\begin{multline*}
\mu^{(t_2)}(A)=
\nu_{t_1/t_2}*\mu^{(t_1)}(A)=
\int_{\mathbb{R}^n}\nu_{t_1/t_2}(A-y)\, \mu^{(t_1)}(dy)=\\
=\int_{\mathbb{R}^n\setminus A^{\epsilon}}\nu_{t_1/t_2}(A-y)\,
\mu^{(t_1)}(dy)+ \int_{A^{\epsilon}}\nu_{t_1/t_2}(A-y)\,
\mu^{(t_1)}(dy)<\epsilon+\mu^{(t_1)}(A^\epsilon), 
\end{multline*}\
where $A-y$
is defined as $\{x\in\sX\,:\,x+y\in
A\}$ and $A^{\epsilon}=\{x\in\sX\,:\,B_\epsilon(x)\cap
A\neq\emptyset\}$.  Thus there exists a probability measure
$\mu^{(t_0)}\in\mathcal{M}\setminus\{0\}$
such that $\mu^{(t)}\Rightarrow\mu^{(t_0)}$
as $t\downarrow
t_0$ implying $\nu_t*\mu^{(t)}\Rightarrow
\nu_{t_0}*\mu^{(t_0)}$. 
Therefore $\nu_{t_0}*\mu^{(t_0)}=\mu$.
Finally we note that $\mu^{(t_0)}\in\widetilde{\SS}$
because of the definition of $t_0$.
\end{proof}

Thanks to Proposition \ref{prop:spectral_decomp}, for any $\mu$ in
$\sM\setminus\{0\}$ there is one and only one couple
$(t,\mu_0)\in (0,1]\times\widetilde{\SS}$ such that
$t\odot_{dt}\mu_0=\mu$, meaning that $\sM\setminus\{0\}$ can be
decomposed as
\begin{align}\label{eq:spectral_decomposition}
\sM\setminus\{0\}\;=&\;(0,1]\times\widetilde{\SS}.
\end{align}\
Any measure $\Lambda$ satisfying \eqref{eq:new_homogeneity} can then
be represented as $\Lambda=\theta_\alpha\otimes\sigma$, where
$\theta_\alpha\left((a,b]\right)=(b^{-\alpha}-a^{-\alpha})$ for any
$(a,b]\subseteq (0,1]$ and
$\sigma(A)=\Lambda\left((0,1]\times A\right)$ for any
$A\in\sB(\widetilde{\SS})$. In fact for any $(a,b]\subseteq (0,1]$ and
$A\in\sB(\widetilde{\SS})$
\begin{multline*}
\Lambda((a,b]\times A)=
\Lambda((0,b]\times A)-\Lambda((0,a]\times A)\stackrel{\eqref{eq:new_homogeneity}}=\\=
b^{-\alpha}\Lambda((0,1]\times A)-a^{-\alpha}\Lambda((0,1]\times A)=
\theta_\alpha((a,b])\;\sigma(A).
\end{multline*}\
Since $\theta_\alpha$ is fixed by $\alpha$, there is a one-to-one
correspondence between Levy measures satisfying
\eqref{eq:new_homogeneity} and spectral measures $\sigma$ on
$\widetilde{\SS}$. Thus a Cox point process driven by the parameter
measure $\xi$ is strictly $\alpha$-stable with respect to the
diffusion-thinning operation, if $\xi$ is $\odot_{dt}$-stable, which,
in turn, can be obtained by choosing an arbitrary spectral measure $\sigma$ on $\widetilde{\SS}$
satisfying
\begin{displaymath}
  \int_{\widetilde{\SS}} \mu(B)^\alpha \sigma(d\mu)<\infty
\end{displaymath}
for all any compact subsets $B\subset \R^n$. And then taking $\xi$
with the Laplace functional  
\begin{displaymath}
  L_\xi[u]=\exp\Bigl\{-\int_{\widetilde{\SS}}
  \langle u,\mu\rangle^\alpha \sigma(d\mu)\Bigr\}\,,
  \quad u\in\bmp.
\end{displaymath}

As an example, consider a finite measure $\hat{\sigma}$ on $\R^n$ and
$\sigma$ its push-forward under the map
$x\in\mathcal{X}\mapsto\delta_x\in\mathcal{M}$.  As shown above, the
homogeneous Levy measure on $\mathcal{M}=(0,1]\times\widetilde{\SS}$
having $\sigma$ as spectral measure is
$\Lambda=\theta_\alpha\otimes\sigma$. Therefore $\Lambda$ is supported
by the following subset of $\mathcal{M}$:
\begin{displaymath}
  \mathcal{Y}:=\{\nu_t(\,\cdot-m):\hspace{1.5mm}t\in(0,1],m\in\mathbb{R}^n\}\subset \mathcal{M}.
\end{displaymath}

\subsection{Stability on $\ZZ_+$ and $\RR_+$}
In all the examples considered so far: thinning stability,
$\mathcal{F}$-stability and an example of a general branching
stability, the operation $\bullet_{dt}$ in Section
\ref{sec:thinning_diffusion}, the corresponding stable point
processes were Cox processes driven by a parameter measure which is
itself stable with respect to a corresponding operation on
measures. Unlike its counterpart operation on point processes, this operation
 was not stochastic, meaning that its result on a deterministic
measure is also deterministic. For the case of thinning stability, this
was an operation of ordinary multiplication, for a point-processes
branching operation $\bullet_{dt}$, this was the operation $\odot_{dt}$.
Nevertheless this does not need to be the case in general: in this
section we consider point-processes branching operations $\bullet$
whose corresponding measure branching operation $\odot$ is
also stochastic.

We consider the case of a trivial phase space $\mathcal{X}$ consisting of one
point, or in other words the case of random variables taking values in $\ZZ_+$ and $\RR_+$.
For $\ZZ_+$, since the phase space consists of one point, the general branching 
stability corresponds to
the $\mathcal{F}$-stability described in Section \ref{sec:das0}. We
introduce the notion of branching stability in $\RR_+$ using the
theory of continuous-state branching processes (CB-processes, see
\cite{Lamperti:67}). We show that, at least in the cases we consider,
branching stable (or $\mathcal{F}$-stable) discrete random variables
are Cox processes driven by a branching stable continuous random
variable. Finally we show how to use quasi-stationary distributions
to construct branching stable continuous random variables.

We argue that the theory of superprocesses (e.g.\ \cite{Etheridge:00})
should be relevant to extend the ideas presented in this section
to general branching stable point processes.

\paragraph{Continuous-state branching processes}
Continuous-state branching processes were first considered in
\cite{Jivrina:58} and \cite{Lamperti:67}, and can be thought of as an
analogue of continuous time branching processes on $\ZZ_+$ on a
continuous space $\RR_+$.
\begin{definition}
  A continuous-state branching process (CB-process) is a Markov
  process $(Z^x_t)_{x,t\geq 0}$ on $\RR_+$, where $t$ denotes time and
  $x$ the starting state, with transition probabilities
  $\left(P_t^x\right)_{t,x\geq 0}$ satisfying the following
  \emph{branching property}:
\begin{equation}\label{eq:branching_property}
P_t^{x+y}\;=\;P_t^x* P_t^y,
\end{equation}\
for any $t,x,y\geq 0$, where $*$ denotes convolution.
\end{definition}
A useful tool to study CB-processes is the spatial Laplace transform $V_t$, defined by
\begin{equation}\label{eq:spatial_laplace_transform}
x\,V_t(z)\quad=\quad -\log\int_{\RR_+}e^{-z\,y}P_t^x(dy)\qquad z\geq 0,
\end{equation}\
for $t\geq0$ and $x>0$. The value of $x$ in
\eqref{eq:spatial_laplace_transform} is irrelevant because of the
branching property, it could be simply set to
1. Using the Chapman-Kolmogorov equations it follows from
\eqref{eq:spatial_laplace_transform} that $(V_t)_{t\in\RR_+}$ is a
composition semigroup
\begin{equation}\label{eq:cb_process_composition}\tag{C1$'$}
V_t(V_s(z))\quad=\quad V_{t+s}(z),\qquad s,t,z\geq 0.
\end{equation}\
Similarly to the discrete case in Section \ref{sec:brachn-proc-refr}, we
focus in the subcritical case, $\EE[Z^1_t]<1$, and we assume
regularity conditions analogous to
\eqref{eq:subcriticality}-\eqref{second_condition_br_proc}. More
specifically, rescaling the time by a constant factor if necessary, we
may assume that
\begin{align}
\EE[Z^1_t]&\;=\;e^{-t}\,,\label{eq:cb_process_subcritical}\tag{C2$'$}\\
\lim_{t\downarrow 0}V_t(z)&\;=\;V_0(z)\;=\;z\,,\label{first_condition_cb_proc}\tag{C3$'$}\\
\lim_{t\rightarrow \infty}V_t(z)&\;=\;0\,.\label{second_condition_cb_proc}\tag{C4$'$}
\end{align}\
\begin{example}\label{ex:cb_process}
  A well known example of CB-process is the diffusion process with
  Kolmogorov backward equations given by
  \begin{displaymath}
\frac{\partial u}{\partial t}\;=\;a x\frac{\partial u}{\partial
  x}+\frac{b x}{2} \frac{\partial^2 u}{\partial x^2}.
\end{displaymath}
The spatial Laplace transform of the corresponding CB-process is
\begin{equation}\label{eq:cb_proc_example}
    V_t(z)\;=\left\{
  \begin{array}{ll}
    \frac{z\,\exp(a t)}{1-(1-\exp(a t))\frac{b z}{2a}}, & \hbox{if }a\neq 0,\\
\frac{z}{1+t\frac{b z}{2}}, & \hbox{if }a =0.
  \end{array}
\right.
\end{equation}\
The sub-critical case corresponds to $a<0$. Rescaling time to satisfy
\eqref{eq:cb_process_subcritical} corresponds to setting $a=-1$.
\end{example}

The \emph{limiting conditional distribution} (or Yaglom distribution)
of $(Z^x_t)_{x,t\geq 0}$ is the weak limit of $(Z^x_t|Z^x_t>0)$, when
$t\rightarrow+\infty$. Such limit does not depend on $x$ (e.g.\
\cite[Th.~3.1]{Lambert2007} or \cite[Th.~4.3]{Li2000}) and we denote
by $Z_\infty$ the corresponding random variable and by $L_{Z_\infty}$
its Laplace transform. The Yaglom distribution is also a
quasi-stationary distribution, meaning that
$\left(Z^{Z_\infty}_t|Z^{Z_\infty}_t>0\right)\stackrel{\mathcal{D}}=Z_\infty$
(e.g.\ \cite[Th.~3.1]{Lambert2007}). Given
\eqref{eq:cb_process_subcritical} and the quasi-stationarity of
$Z_\infty$ it follows that
\begin{equation}\label{eq:cb_process_quasi_stationary}
  L_{Z_\infty}\big(V_t(z)\big)=1-e^{-t}+e^{-t}L_{Z_\infty}(z),\qquad s,z\geq 0\,.
\end{equation}\

\paragraph{$\sV$-stability and Cox characterisation of
  $\mathcal{F}$-stable random variables}
\label{sec:V_stability}
Let $(Z^x_t)_{x,t\geq 0}$ be a CB-process with spatial Laplace
transform $\sV=(V_t)_{t\geq 0}$, satisfying assumptions
\eqref{eq:cb_process_composition}-\eqref{second_condition_cb_proc} of
the previous section. Define a corresponding stochastic operation
acting on random variables on $\RR_+$ as follows:
\begin{equation}\label{eq:v_multiplication}
t\odot_\sV \xi \;\stackrel{\mathcal{D}}=\;Z^\xi_{-\ln t}\qquad 0< z\leq 1,
\end{equation}\
where $\xi$ is an $\RR_+$-valued random variable and $Z^\xi_t$ is the
CB-process with random starting state $\xi$.  Similarly to Proposition
\ref{maximality_of_the_operation}, from the Markov and branching
properties of $(Z^x_t)_{x,t\geq 0}$ it follows that $\odot_\sV$ is
associative and distributive with respect to the usual sum.

The notion of $\sV$-stability for continuous random variables is
analogous to the notion of $\sF$-stability for discrete frameworks:
\begin{definition}
  A $\RR_+$-valued random variable $X$ (or its distribution) is
  $\sV$-stable with exponent $\alpha$ if
\begin{equation}\label{eq:V_stability}
  t^{1/\alpha}\odot_\sV X'+(1-t)^{1/\alpha}\odot_\sV X''\;\stackrel{\mathcal{D}}=\;X\qquad 0< t< 1,
\end{equation}\
where $X'$ and $X''$ are independent copies of $X$.
\end{definition}

In terms of Laplace transform $L$, the definition of $\odot_\sV$ in
\eqref{eq:v_multiplication} can be written as
$L_{t\odot_\sV X}(z)=L_X(V_{-\ln t}(z))$. Thus, arguing as in
Proposition~\ref{definitions_of_stability}, \eqref{eq:V_stability}
is equivalent to
\begin{equation}\label{eq:V_stability_laplace}
L_X(V_{-\ln t}(z))\;=\;L_X(z)^{t^{\alpha}}
\qquad 0< t< 1.
\end{equation}
Suppose we have a continuous-time branching process on $\ZZ_+$ with
p.g.f.'s $\sF=(F_t)_{t\geq 0}$ and a CB-process with spatial Laplace
transform $\sV=(V_t)_{t\geq 0}$ such that
\begin{equation}\label{eq:discrete_continuous}
F_t(z)\;=\; 1-V_t(1-z)\qquad 0\leq z\leq 1.
\end{equation}\
The relation between the discrete and continuous stochastic
operations, $\circ_\sF$ and $\odot_\sV$, is that
Cox random variables 
driven by a $\sV$-stable random intensity are $\sF$-stable.
\begin{proposition}\label{prop:cox_characterization_f_stability}
  Let $\sF=(F_t)_{t\geq 0}$ and $\sV=(V_t)_{t\geq 0}$ satisfy
  \eqref{semigroup_property}-\eqref{second_condition_br_proc} and
  \eqref{eq:cb_process_composition}-\eqref{second_condition_cb_proc}
  respectively, and let \eqref{eq:discrete_continuous} be satisfied.
  Let $\xi$ be a $\sV$-stable random variable with exponent $\alpha$
  and let $X$ be a Cox random variable driven by $\xi$, meaning that \
  $X|\xi\sim Po(\xi)$. Then $X$ is $\sF$-stable with exponent
  $\alpha$.
\end{proposition}
\begin{proof}
The pg.f.\ of $X$ is given by $G_X(z)=L_\xi(1-z)$, see e.g.\
\eqref{eq:pgfl_cox}. Therefore 
\begin{equation*}
G_X(F_{-\ln t}(z))
=
L_\xi(1-(F_{-\ln t}(z)))
\stackrel{\eqref{eq:discrete_continuous}}=
L_\xi(V_{-\ln t}(1-z))
=
L_\xi(1-z)^{t^\alpha}
=
G_X(z)^{t^\alpha},
\end{equation*}\
which implies that $X$ is $\sF$-stable with exponent $\alpha$.
\end{proof}

Examples of discrete and continuous operations, $\circ_\sF$ and
$\odot_\sV$, with $\sF$ and $\sV$ coupled by
\eqref{eq:discrete_continuous} are thinning and multiplication, as
well as the birth and death process of Example~\ref{ex:branching_geom} and
the CB-process of Example~\ref{ex:cb_process} (with $b=1$). Also the
operations $\bullet_{dt}$ and $\odot_{dt}$ of
Section~\ref{sec:thinning_diffusion}, in a point process and random
measures framework, satisfy \eqref{eq:discrete_continuous}. A natural
question is whether for any continuous time branching process on
$\ZZ_+$ there is a CB-process such that \eqref{eq:discrete_continuous}
is satisfied and viceversa.
Note that, given \eqref{eq:discrete_continuous}, $(F_t)_{t\geq 0}$ is
a composition semigroup if and only if $(V_t)_{t\geq 0}$ is. Indeed,
\begin{displaymath}
F_t(F_s(z))\;=\; 1-V_t(1-(1-V_s(1-z)))\;=\; 1-V_t(V_s(1-z))\;=\; 1-V_{t+s}(1-z)=F_{t+s}(z),
\end{displaymath}
and similarly for $V_t$. Therefore one would only need to prove that if
$V_t$ is the spatial Laplace transform of a random variable on $\R_+$
then $F_t$ defined by \eqref{eq:discrete_continuous} is the p.g.f.\ of
a random variable on $\ZZ_+$ or viceversa.

Finally, note that Proposition
\ref{prop:cox_characterization_f_stability} suggests that, at least in
some cases, $\sF$-stable random variables on $\ZZ_+$ are Cox processes
driven by $\sV$-stable random variables on $\RR_+$. It is therefore
natural to ask whether we can characterise $\sV$-stable random
variables.

Equations \eqref{eq:cb_process_quasi_stationary} and
\eqref{eq:v_multiplication} imply that
\begin{equation}\label{eq:cb_proc_vs_thinning}
  t\odot_\sV Z_\infty\deq \sum_{i=0}^{t\circ 1}\,Z_\infty\qquad t\in(0,1],
\end{equation}\
where $t\circ 1$ is, by the definition of thinning, a binomial
$\Bin(1,t)$ random variable (independent of $Z_\infty$). Therefore the
Yaglom distribution allows us to pass from $\odot_\sV$ to thinning and
use such a property to construct $\sV$-stable random variables
from $\das$ random variables.

\begin{proposition}
  Let $X$ be a $\das$ random variable on $\ZZ_+$ (see \eqref{das0}),
  and $\xi=\sum_{i=1}^{X}Z_{\infty}^{(i)}$, where
  $Z_{\infty}^{(1)},Z_{\infty}^{(2)},\dots$ are i.i.d.\ copies of the
  Yaglom distribution $Z_\infty$. Then $\xi$ is $\sV$-stable with
  exponent $\alpha$.
\end{proposition}
\begin{proof}
  Given its definition, the Laplace transform of $\xi$ is given by
  $L_\xi(z)=G_X\left(L_{Z_\infty}(z)\right)$. Therefore
\begin{multline*}
L_\xi(V_{-\ln t}(z))
\;=\;
G_X\left(L_{Z_\infty}(V_{-\ln t}(z))\right)
\;\stackrel{\eqref{eq:cb_process_quasi_stationary}}=\;
G_X\left(1-t+tL_{Z_\infty}(z)\right)
\;=\\
G_{t\circ X}\left(L_{Z_\infty}(z)\right)
\;\stackrel{\eqref{defi_stability4}}=\;
G_X\left(L_{Z_\infty}(z)\right)^{t^{\alpha}}
\;=\;
L_\xi(z)^{t^{\alpha}},
\end{multline*}
which implies that $\xi$ if $\sV$-stable.
\end{proof}

\section{Discussion}
In this paper we have studied discrete stability with respect to the
most general branching operation on counting measures which unifies
all notions considered so far: discrete stable and $\mathcal{F}$-stable
integer random variables, thinning-stable and $\mathcal{F}$-stable
point processes characterised above. We considered in detail an
important example of thinning-diffusion branching stable point
processes and established the corresponding spectral representation of
their laws. We demonstrate that branching stability of integer random
variables may be associated with a stability with respect to a
stochastic operation of continuous branching on the positive real line and we
conjecture that this association may still be true in general for
point processes and its continuous counterpart, random measures. A
full characterisation of the branching-stable point processes, as well
as of the associated stable random measures is yet to be established. 

\section*{Acknowledgement}
\label{sec:acknowledgement}
The authors are thankful to Ilya Molchanov for fruitful
discussions and to the anonymous referee for thorough reading of the
manuscript and numerous suggestions which significantly improved
its exposition. SZ also thanks Serik Sagitov and Peter
Jagers for consultations on advanced topics in branching processes.


\providecommand{\bysame}{\leavevmode\hbox to3em{\hrulefill}\thinspace}
\providecommand{\MR}{\relax\ifhmode\unskip\space\fi MR }
\providecommand{\MRhref}[2]{%
  \href{http://www.ams.org/mathscinet-getitem?mr=#1}{#2}
}
\providecommand{\href}[2]{#2}

\end{document}